\documentclass[12pt]{article}
\usepackage[T1]{fontenc}
\usepackage[utf8]{inputenc}
\usepackage[english]{babel}

\usepackage{xcolor}
\usepackage[bottom=2cm,top=2cm,left=2cm,right=2cm]{geometry}
\usepackage{amssymb, amssymb, amsmath, amsfonts, mathrsfs, amsthm}

\usepackage{fancybox}

\usepackage{mathpazo}
\numberwithin{equation}{section}

\newcommand{\X}{\mathsf{X}}

\theoremstyle{plain}
\newtheorem{theorem}{Theorem}[section]

\newtheorem{lemma}[theorem]{Lemma}
\newtheorem{definition}[theorem]{Definition}

\newtheorem{remark}{Remark}[section]

\newtheorem{step}{Step}[theorem] 
 	
\renewenvironment{proof}{$\mathbf{Proof}$.}{\begin{flushright}\qedsymbol\end{flushright}}
\title{ {Approximate Controllability
of Nonlocal  Stochastic Integrodifferential System in Hilbert Spaces}}
\author{Mamadou Pathe LY$^{a}$ \footnote{E-mail address: moha.pathe.ly@gmail.com},
	\:Ravikumar Kasinathan$^{b}$ \footnote{E-mail address : ravikumarkpsg@gmail.com},
    \:Ramkumar Kasinathan$^{b}$\footnote{E-mail address: ramkumarkpsg@gmail.com},\\
    \:Dimplekumar Chalishajar$^{c,d}$\footnote{E-mail address: chalishajardn@vmi.edu(Corresponding author)},
	\:Mamadou Abdoul Diop$^{a,d}$\footnote{E-mail address:  mamadou-abdoul.diop@ugb.edu.sn},
}
\begin{document}
	\renewenvironment{proof}{$\mathbf{Proof}$.}{\begin{flushright}\qedsymbol\end{flushright}}
	\newenvironment{prob}{$\mathbf{Problem}$.}{\emph{\begin{flushright}\qedsymbol\end{flushright}}}
	\newenvironment{problem}[1]{\par\noindent\underline{Problem:}\space#1}{}
	\providecommand{\keywords}[1]{\textbf{\textit{Keywords:}} #1}
	\providecommand{\classification}[1]{\textbf{\textit{Mathematcal Subject Classification:}} #1}
	%%%%%%%%%%%%%%%%%%%%%%%%%%%%%%%%%%%%%%%%%%%%%%%%%%%%%%%%%%%%%%%%%%%%%%%%%
	% \renewenvironment{proof}{$\mathbf{Proof}$.}{\begin{flushright}\qedsymbol\end{flushright}}
	\selectlanguage{english}
	\maketitle
	{\footnotesize
		\begin{flushleft}			
			\thanks{$^a$ Universit\'e Gaston Berger de Saint-Louis, \,UFR SAT D\'epartement de Math\'ematiques, B.P234,\,Saint-Louis, S\'en\'egal.}\\
			\thanks{$^b$ Department of Mathematics, PSG College of Arts \& Science, Coimbatore, 641 014, India.}\\
		  \thanks{$^c$ Department of Applied Mathematics, Mallory Hall, Virginia Military Institute (VMI), Lexington, VA 24450, USA.}\\
           % \thanks{$^b$ Universidade Technológica Fedral do Paraná, UTFPR}\\
			\thanks{$^d$ UMMISCO UMI 209 IRD/UPMC, Bondy, France.}
	\end{flushleft}}
	%%%%%%%%%%%%%%%%%%%%%%%%%%%%%%%%%%%%%%%%%%%%%%%%%%%%%%%%%%%%%%%%%%%%%%%%%
	% \maketitle	
	\begin{abstract}
	\noindent This project  investigates the approximate controllability of a class of stochastic integrodifferential equations in Hilbert space with non-local beginning conditions. In a departure from the conventional concerns expressed in the literature, we will not consider compactness or the Lipschitz criteria concerning the nonlocal term. \textcolor{blue}{We use the fact that the resolvent operator is compact}. We first prove the controllability of the nonlinear system using Schauder's fixed point theorem, a method known for its robustness; as well, we also use Grimmer's resolvent operator theory. Subsequently, we employ the reliable approximation methods and the powerful diagonal argument to determine the approximate controllability of the stochastic system. To conclude, we present an example that validates our theoretical statement.
    %In this work, we study the approximate controllability of a class of stochastic integrodifferential equations in Hilbert space with non-local initial conditions. We do not consider compactness nor the Lipschitz criteria on the non-local term, contrary to the usual concerns found in the literature. In addition to Grimmer's resolvent operator theory, we first establish, by means of Shauder's fixed point theorem, the controllability of the nonlinear system. Secondly, we analyze the approximate controllability of the stochastic system using approximation approaches and the diagonal argument. We provide an example to demonstrate the abstract theory towards the conclusion of our work.
	\end{abstract}	
    \\
	\classification{		
	%26A42,
	%34A38,
        34K50,
        45D05,
        34H05,
        37N35
.}	
\section{Introduction}
%\par Based on the work of \cite{sakthivel2013approximate},
This paper seeks to provide the necessary criteria for the approximate controllability of abstract nonlinear integrodifferential equations with nonlocal initial conditions powered by a cylindrical Wiener process of the following form: 
 \begin{equation}\label{sys}
	\begin{cases}
\vartheta^\prime(\sigma)=A\vartheta(\sigma)+\displaystyle\int_0^\sigma \Pi(\sigma-s)\vartheta(s)\mathrm{d}s  + Cu(\sigma) + f(\sigma, \vartheta(\sigma))+ \textcolor{blue}{g(\sigma, \vartheta(\sigma))\mathrm{d}\mathsf{W}(\sigma), \; \sigma}\in J=[0,c]\\ 
	\vartheta(0) =  h(\vartheta), 
	\end{cases}
\end{equation}
where the state $\vartheta(\cdot)$ takes values in the Hilbert space $\mathbb{H}$, endowed with inner product $\langle\cdot, \cdot\rangle$ and norm $\|\cdot\|$; the operator $A:\mathcal{D}(A)\subset\mathbb{H}\to\mathbb{H}$ is the infinitesimal generator of a strongly continuous semigroup $(T(\sigma))_{\sigma\geq 0}$ in $\mathbb{H}$; for $\sigma\geq 0, \Pi(\sigma)$ is a closed linear operator with domain $\mathcal{D}(\Pi(\sigma))\supset \mathcal{D}(A)$ independent of $\sigma$. The nonlocal function $h$ is defined by 
$$
    h(\vartheta) = \int_0^c\zeta(s,\vartheta(s))\mathrm{d}s.
$$ 
Consider $\mathbb{Y}$ to be an additional separable Hilbert space powered by the same scalar product and norm as $\mathbb{H}$ (without any confusion), the set $\{\mathsf{W}(\sigma): \sigma\geq 0\}$ is a given $\mathbb{Y}$-valued Wiener process with a finite trace nuclear covariance operator $\mathcal{Q}\geq 0$ defined on a complete filtered probability space $(\Omega, \mathcal{F}, \{\mathcal{F}_\sigma\}_{\sigma\geq 0}, \mathbb{P})$, with $ \{\mathcal{F}_\sigma\}_{\sigma\geq 0}$ generated by the Wiener process $\mathsf{W}$. The control functions $u(\cdot)$ belong to the Banach space of admissible control functions space $L_\mathcal{F}^2(J, \mathbb{Y})$, and $C:\mathbb{Y}\to\mathbb{H}$ is a bounded linear operator. Later, the nonlinear functions $f, g$, and $\zeta$ will be defined.

Stochastic differential equations are used in fields such as engineering, economics, biology, electricity, etc., to model the noise effects that occur when studying the phenomena that appear in these fields (see \cite{mao2007stochastic,sakthivel2013existence, zhang2019mild}, Da Prato and J. Zabesczyk book \cite{da2014stochastic} and references therein). The study of the qualitative properties like existence, stability, optimal control, and controllability, among others, of these kinds of systems, has attracted great interest in the researcher and engineers community; for instance, we refer to \cite{dauer2004controllability} and \cite{farahi2014approximate} as references. 

On the other hand, controllability is an important concept in control theory, both in deterministic and stochastic cases. In infinite dimension, we are always tempted to achieve exact controllability, i.e., to be able to steer a system between two arbitrary points $v_0, v_1$ in state space, as we are able to do in finite dimension. However, this is rarely satisfied when the semigroup associated with a linear system is compact since the controllability operator is asked to be surjective, see \cite{curtain2012introduction, triggiani1977note} to learn more about this. Therefore, we need to switch from exact controllability to a weaker one, which is approximate controllability. In recent years, much work has been done in the study of approximate controllability of fractional differential systems in the deterministic and stochastic cases; see \textcolor{blue}{\cite{mahmudov2003approximate, sakthivel2011approximate, sakthivel2012approximate, sivasankar2022, gokul2024}}. Sakthivel et al.\cite{sakthivel2011approximate} established the approximate controllability by assuming that the $C_0$-semigroup is compact and the nonlinear function is continuous and uniformly bounded. \textcolor{blue}{More recently, K. Nandhaprasadh, R. Udhayakumar \cite{Nandhaprasadh2024}, discussed the approximate  boundary controllability of Hilfer  fractional neutral stochastic differential inclusions with  fractional Brownian motion (fBm) and Clarke’s subdifferential in  Hilbert space, and examine whether mild solutions to a fractional stochastic evolution system with a fractional Caputo derivative on an infinite interval exist and are attractive, and in \cite{sivasankar2025}, S. Sivasankar et al., examine whether mild solutions to a fractional stochastic evolution system with a fractional Caputo derivative on an infinite interval exist and are attractive}. However, little has been done on the study of stochastic integrodifferential systems by means of the Grimmer resolvent operator properties. Diop et al.\cite{diop2022optimal} studied the existence of optimal controls for some impulsive stochastic integrodifferential equations with state-dependent delay.

Moreover, the nonlocal property represents one of the most important tools to consider when using integrodifferential in the applications we have described above. Indeed, nonlocal initial conditions mean that the current state of a system depends not only on its previous state but also on all its previous states. It also means that, contrary to the abstract differential equation, where the initial condition is taken to be $\vartheta(0)=\vartheta_0$, in order to check the state of nonlocal equations at a given point, information about the values of the system far from that point is needed; one can see \cite{deng1993exponential} for more explanation about the importance of the nonlocal initial condition property. Differential equations with the nonlocal property have been investigated by many authors \cite{aizicovici1997functional, deng1993exponential}, and in the paper of Byszewski \cite{byszewski1991theorems}, the study of differential equations by taking into account the nonlocal property has been introduced for the first time. But more often, the restrictions, such as compactness or the Lipschitz condition, that are imposed on the nonlocal term $h$ in the literature are too strong and are not easily satisfied in practical applications. To relax these conditions on $h$, Liang et al. \cite{liang2004nonlocal} introduced the following assumption:
\begin{align*}
\begin{array}{c c l}
    \textbf{(H0)} \mbox{ For any } x, y\in C([0, c]; X), \mbox{ there exist a constant } \beta\in (0, c) \mbox{ such that }\\ x(\sigma)=y(\sigma)(\sigma\in [\beta, c]) \mbox{ implies } h(x)=h(y)  
\end{array}
\end{align*}
This means that the values of the solution $\vartheta(\sigma)$ when $\sigma$ takes zero do not affect $h$. In this work, contrary to what is supposed in the above hypothesis \textbf{(H0)}, we consider that the nonlocal term $h(\vartheta)$ depends on all values of the solution $\vartheta(\cdot)$ on the whole interval $[0, c]$. In  \cite{ding2020approximat}, Yonghong Ding and Yongxiang Li study the approximate controllability of fractional stochastic evolution equations with nonlocal conditions. Therefore, based on this work, \cite{ding2020approximat}, with the help of the compactness of the resolvent operator $\Re(\sigma)$, the techniques of stochastic analysis, approximation technique, diagonal argument, and Schauder's fixed point theorem, to establish the approximate controllability results, considering weaker assumptions like continuity of $h$.

The main contributions of this paper are summarized in the following points:
\begin{itemize}
    \item[$\bullet$] A new result on the approximate controllability of nonlocal stochastic integrodifferential equations in Hilbert space is obtained.
    \item[$\bullet$] We relax the compactness conditions on the nonlocal term and use the Grimmer resolvent operator and the fixed point theory to establish the existence of a mild solution for system \eqref{sys}.
    \item[$\bullet$] Finally, we use approximation techniques and the diagonal argument to prove the approximate controllability of our system \eqref{sys}. 
\end{itemize}
\noindent The rest of this work is organized as follows. In Section \ref{sec1}, we recall some notions and basic known facts that will help us establish our results. In the section \ref{sec2}, we define the notion of a mild solution and give sufficient conditions that ensure its existence. In Section \ref{sec3}, we establish the approximate controllability of the system (\ref{sys}). In the section \ref{sec4}, we give an example to illustrate the feasibility of our obtained results. 
\section{Preliminaries results}\label{sec1}
In this section, we present the well-known essential facts, basic definitions, lemmas, and preliminary results that are used throughout this paper.
\subsection{Wiener Process}
All through this work we consider $(\Omega, \mathcal{F}, \left\{\mathcal{F}_\sigma\right\}_{0\le \sigma\le c}, \mathbb{P})$ with $c>0$ an arbitrarily fixed time horizon, as a filtered complete probability space satisfying the usual conditions, which means that the filtration is a right continuous increasing family and $\mathcal{F}_0$ contains all $\mathbb{P}$-null sets of the filtration $\mathcal{F}$. We also consider $\mathbb{H}$ and $\mathbb{Y}$ as two real separable Hilbert spaces, with $\langle\cdot, \cdot\rangle_\mathbb{H},\langle\cdot, \cdot\rangle_\mathbb{Y}$ their inner product and $\|\cdot\|_\mathbb{H}, \|\cdot\|_\mathbb{Y}$ their corresponding vector norms, respectively. Let $\left\{e_k\right\}_{k\in\mathbb{N}}$ be a complete orthonormal basis of $\mathbb{Y}$ consisting of eigenvectors of $\mathcal{Q}$ corresponding to the eigenvalues $\lambda_k, k\in\mathbb{N}$. Suppose that the set $\left\{\mathsf{W}(\sigma):\sigma\geq 0\right\}$ is a cylindrical $\mathbb{Y}-$valued Brownian motion or Wiener process defined on the filtered probability space $(\Omega, \mathcal{F}, \left\{\mathcal{F}_\sigma\right\}_{\sigma\geq 0}, \mathbb{P})$ with finite trace nuclear covariance operator $\mathcal{Q}\geq 0$ denoted by: $$Tr(\mathcal{Q})=\sum_{k=1}^\infty\lambda_k = \lambda< + \infty,$$ which satisfies that $\mathcal{Q} e_k = \lambda_k e_k, k\in\mathbb{N}$. Let $\left\{\gamma_k(\sigma), k\in\mathbb{N}\right\}$ be a sequence of one-dimensional standard Wiener processes mutually independent on $(\Omega, \mathcal{F}, \left\{\mathcal{A}\right\}_{\sigma\geq 0}, \mathbb{P})$ such that for $\sigma\geq 0$, 
\begin{equation}\label{wieserie}
\mathsf{W}(\sigma) = \sum_{k=1}^\infty\sqrt{\lambda_k}\gamma_k(\sigma)e_k, 
\end{equation}
where 
$$
\gamma_k(\sigma) = \frac{1}{\sqrt{\lambda_k}}\langle\mathsf{W}(\sigma), e_k\rangle, k\in\mathbb{N},
$$
represents a pair of independent real-valued Brownian motions on $(\Omega, \mathcal{F}, \left\{\mathcal{F}_\sigma\right\}_{\sigma\geq 0}, \mathbb{P})$(see \cite{da2014stochastic} for more details on the series \eqref{wieserie}).\\
%Furthermore, we assume that $\mathcal{F}_t =\sigma\left\{W(s), 0\leq s\leq t\right\}$ is the $\sigma-$algebra generated by $W$. \\
Furthermore, we review the meaning of the $\mathbb{H}$-valued stochastic integral in relation to the $\mathbb{Y}$-valued $\mathcal{Q}$-Wiener process $\mathsf{W}$. Let $L_2^0 = L_2(\mathcal{Q}^{\frac{1}{2}}(\mathbb{Y}), \mathbb{H})$ represent the space of all Hilbert-Schmidt operators from $\mathcal{Q}^{\frac{1}{2}}Y$ into $\mathbb{H}$ with the norm $$\|\phi\|_{L_2^0}^2 = Tr((\phi \mathcal{Q}^{1/2})(\phi\mathcal{Q}^{1/2})^*),$$ 
where $\phi^*$ is the adjoint of the operator $\phi\in L_2^0$. We observe that, here, this norm can be reduced to 
$$\|\phi\|_{L_2^0}^2 = Tr(\phi \mathcal{Q}\phi^*)=\sum_{k=1}^\infty\|\sqrt{\gamma_k}\phi_k\|^2.
$$ 
%The collection of all $\mathcal{F}_b-$measurable, square integrable $\mathbb{H}$-valued random variables denoted $L^2(\Omega,\mathbb{H})$, is a Banach space equipped with the norm $||v||_{L^2(\Omega,\mathbb{H})}=(\mathbb{E}||v(t)||)^{1/2}$, where $\mathbb{E}$ denotes the expectation with respect to the measure $\mathbb{P}$. 

\noindent Let $\mathcal{C}(J, L^2(\Omega,\mathbb{H}))$ represent the Banach space of all continuous mappings from $J$ to $L^2(\Omega,\mathbb{H})$ that fulfill $\sup_{\sigma\in J}(\mathbb{E}||\vartheta(\sigma)||^2) <\infty$. We refer to $\mathcal{M}(J,L^2(\Omega,\mathbb{H}))$ as the space of all $\mathcal{F}_\sigma$-adapted measurable processes $\vartheta\in\mathcal{C}(J, L^2(\Omega,\mathbb{H}))$ endowed with the norm $||\vartheta||_{\mathcal{M}}=(\mathbb{E}||\vartheta(\sigma)||)^{1/2}$.\\

\noindent The result that follows will be used to calculate the system's stochastic integral.
\begin{lemma}\cite{curtain1971stochastic}\label{lemma21}
    If $g:J\times\mathbb{H}\to L(Y, \mathbb{H})$ is  continuous function and $\vartheta\in\mathcal{C}(J, L^2(\Omega, \mathbb{H}))$, then 
    $$
    \mathbb{E}\|\int_J g(\sigma, \vartheta(\sigma))\mathrm{d}\mathsf{W}(\sigma)\|^2\leq Tr(\mathcal{Q})\int_J\mathbb{E}\|g(\sigma,\vartheta(\sigma))\|^2\mathrm{d}\sigma.
    $$
\end{lemma}
\subsection{Integrodifferential Equation}
%In what follows, we refer to the space of all continuous functions from $J$ into $\X$ with the sup-norm $$||v||_{\mathcal{C}}=\sup_{t\in J}||v(t)||,\ \text{for}\ v\in\mathcal{C}.$$ as $\mathcal{C}$.
%Let's consider $A$ and $\Pi(\sigma)$ as closed linear operators on $\mathbb{H}$. 
By defining $Y$ as a Banach space, $\mathcal{D}(A)$ powered with the graph norm given by $$\|y\|_Y := \|Ay\| + \|y\|\  \text{for}\ y\in Y.$$  
%In addition, consider the notations $C([0,+\infty); Y), B(Y,\X)$ that stand for the space of all continuous functions from $[0, +\infty)$ into $Y$, the set of all bounded linear operators from $Y$ into $\X$, respectively. 
Let's take into account the following Cauchy problem:
\begin{equation}\label{cauchy}
\left\{\begin{array}{ll}
\vartheta^\prime(\sigma)= A\vartheta(\sigma) + \displaystyle\int_0^\sigma \Pi(\sigma-s)\vartheta(s)ds \ \text{for}\ \sigma\geq 0 \\ 
\vartheta(0)= \vartheta_0\in \X.
\end{array}\right.
\end{equation}
To deal with a resolvent operator, we suppose that $A$ and $\Pi(\cdot)$ meet the following criteria:
\begin{enumerate}
	\item[(\textbf{R1})] $A$ is the infinitesimal generator of a $C_0$-semigroup $T(\sigma), \sigma\geq 0$ of uniformly bounded operators in $\X$.
	\item[(\textbf{R2})] For all $\sigma\geq 0, \Pi(\sigma)$ is a closed linear operator from $\mathcal{D}(A)$ to $\X$ and $\Pi(\sigma)\in B(Y, \X).$ For any $y\in Y$, the map $\sigma\to \Pi(\sigma)y$ is bounded, differentiable, and the derivative $\sigma\to \Pi'(\sigma)y$ is bounded uniformly continuous on $\mathbb{R}^+$. 
\end{enumerate}
\begin{definition}\cite{grimmer1982resolvent}
Let $\vartheta_0\in Y$. A solution $\vartheta(\cdot)$ of  \eqref{cauchy} is a function that belongs to $C([0, \infty), Y)\cup C^1([0, \infty), \X)$ so that $\vartheta(0) = \vartheta_0$ and \eqref{cauchy} is satisfied for all $\sigma\ge 0$.
\end{definition}
The definition of a resolvent operator (in the sense of Grimmer) follows logically from the definition given above.
\begin{definition}\cite{grimmer1982resolvent}\label{def21}
	A resolvent for Eq.(\ref{cauchy}) is a bounded linear operator-valued function $\Re(\sigma)\in B(\X)\ \text{for}\ \sigma\geq 0$, having the following properties :
	\begin{enumerate}
		\item[(a)] $\Re(0)=I$ and $|\Re(\sigma)|\leq Me^{\beta \sigma}$ for some constants $M>0$ and $\beta\in\mathbb{R}$.
		\item[(b)] For each $\vartheta\in \X, \Re(\sigma)\vartheta$ is continuous for $\sigma\geq 0$.
		\item[(c)] $\Re(\sigma)\in B(Y)$ for $\vartheta\in Y, \Re(\cdot)\vartheta\in \mathcal{C}^1([0,+\infty); \X)\cap \mathcal{C}([0; +\infty); Y)$ and 
		$$\begin{array}{r c l}
		\Re'(\sigma)\vartheta &=& A\Re(\sigma)\vartheta +\displaystyle\int_0^\sigma\Pi(\sigma-s)\Re(s)\vartheta\mathrm{d}s\\ 
		&=& \Re(\sigma)Av +\displaystyle\int_0^\sigma \Re(\sigma-s)\Pi(s)\vartheta\mathrm{d}s,\ \sigma\geq 0,
		\end{array}$$
	\end{enumerate}
\end{definition}
\noindent We recommend reading \cite{desch1984some, grimmer1982resolvent} and the reference therein, for more information about the resolvent operators and some of their most crucial characteristics.
\begin{theorem}\cite{grimmer1982resolvent} Assume that $(\textbf{H1})-(\textbf{H2})$ hold. Then there exists a unique resolvent operator for the Cauchy problem \eqref{sys}.
\end{theorem}
\noindent In the following, we give some results for the existence of solutions for the following integrodifferential equation:
\begin{equation}\label{IE}
\begin{cases}
\vartheta^\prime(\sigma)=Av(\sigma)+\displaystyle\int_0^\sigma\Pi(\sigma-s)\vartheta(s)ds + q(\sigma)\quad \ \text{for}\; \sigma\ge 0\\
\vartheta(0)=\vartheta_0\in \X, 
\end{cases}
\end{equation}
where $q :\mathbb{R}^+\rightarrow \X$ is a continuous function.
\begin{definition}\cite{grimmer1982resolvent} A continuous function $\vartheta :\mathbb{R}^+\rightarrow \X$ is a strict solution of equation \eqref{IE} if :
\begin{enumerate}
\item $\vartheta\in \mathcal{C}^1(\mathbb{R}^+;\X)\bigcap \mathcal{C}(\mathbb{R}^+; Y)$ and
	\item $\vartheta$ satisfies equation \eqref{IE}.
\end{enumerate}	
\end{definition}
\begin{theorem}\cite{grimmer1982resolvent} 
Assume that (\textbf{H1})-(\textbf{H2}) is verified. If $\vartheta$ is a strict solution of equation \eqref{IE}, then $$ \vartheta(\sigma)=\Re(\sigma)\vartheta_0+\displaystyle\int_0^\sigma \Re(\sigma-s)q(s)ds\quad for\; \sigma\geq 0.$$
\end{theorem}
\begin{lemma}\cite{desch1984some}\label{lemma26}
	Assume that (\textbf{R1})-(\textbf{R2}) is satisfied. The resolvent operator $(\Re(\sigma))_{\sigma\geq 0}$ is compact for $\sigma>0$ if and only if the semigroup $(T(\sigma))_{\sigma\geq 0}$ is compact for $\sigma>0$.
\end{lemma} 
\begin{lemma}\cite{grimmer1982resolvent}\label{lemma27} 
    Assume that (\textbf{R1})-(\textbf{R2}) is satisfied. If the resolvent operator $(\Re(\sigma))_{\sigma\geq 0}$ is compact for $\sigma >0$, then it is norm continuous (or continuous in the uniform operator topology) for $\sigma>0$. 
\end{lemma}
\begin{lemma}\cite{desch1984some}\label{lemma28} For all $c>0$, there exists a constant $\gamma=\gamma(c)$ such that 
    \begin{equation*}
    || \Re(\sigma+ \epsilon)-\Re(\epsilon)\Re(\sigma)||_{\X}\leq \gamma \epsilon ,\quad \text{for}\; 0 \leq \epsilon \leq \sigma\leq c. 
    \end{equation*}
\end{lemma}
%Therefore, we have the following definition,
In the following, we give the related definition of approximate controllability.
\begin{definition}
    The admissible set of system \eqref{sys} at the terminal time $c$ is then referred to as $$\mathcal{R}(c)=\bigg\{\vartheta(c, u): u\in L_\mathcal{F}^2(J, \mathbb{Y})\bigg\},$$ where $\vartheta(c; u)$ denotes the state value of system \eqref{sys} at the terminal time $c$ which corresponds to the control $u$.
\end{definition}
\noindent The above leads to the definition of approximate controllability of system \eqref{sys}.
\begin{definition}\cite{sakthivel2012approximate} The stochastic nonlocal integrodifferential system \eqref{sys} is said to be approximately controllable on the interval $J$ if we have $\overline{\mathcal{R}(c)}= \mathbb{H}$. That is, for all $\epsilon>0$ and every desired final state $\vartheta_1\in \mathbb{H}$, there exists a control $u\in L_\mathcal{F}^2(J, \mathbb{Y})$ such that $\vartheta$ satisfies $\|\vartheta(c) - \vartheta_1\|<\epsilon$.
\end{definition}
\noindent In order to determine the approximate controllability result of system \eqref{sys}, we take into account the approximate controllability of the subsequent  linear integrodifferential system:
 \begin{equation}\label{linsys}
	\begin{cases}
	\vartheta'(\sigma)=A\vartheta(\sigma)+\displaystyle\int_0^\sigma \Pi(\sigma-s)\vartheta(s)\mathrm{d}s  + \mathsf{C}u(\sigma) \quad \text{for}\; \sigma\in J=[0, c]\\ 
	\vartheta(0) =  \vartheta_0, 
	\end{cases}
\end{equation}
and define the following controllability and resolution operators,
 \begin{equation*}
    \begin{cases}
    \Delta_0^c = \displaystyle\int_0^c \Re(c-s)\mathsf{C}\mathsf{C}^*\Re^*(c-s)\mathrm{d}s\\ 
    S(\mu, \Delta_0^c) = (\mu I + \Delta_0^c)^{-1}, \mu>0,
    \end{cases}
\end{equation*}
where the operators $\mathsf{C}^*$ and $\Re^*(\sigma)$ denote  the adjoint of $\mathsf{C}$ and $\Re(\sigma),$ respectively. It is commonly known that the operator $\Delta_0^c$ is linearly bounded and that $||S(\mu, \Delta_0^c)||\leq\frac{1}{\mu}$.
\begin{lemma}\cite{mahmudov2014approximate}\label{lemma33}
    The linear system \eqref{linsys} is said to be approximately controllable on $[0, c]$ if and only if, $\mu S(\mu, \Delta_0^c)\to 0$ when $\mu\to 0$ in the strong operator topology.
\end{lemma}
\begin{lemma}\cite{dauer2004controllability}\label{lemma211}
     For any $\tilde{\vartheta}_c\in L^2(\Omega, \mathbb{H})$, there exists a function $\phi\in L_\mathcal{F}^2(\Omega;  L^2(J, L_2^0))$ so that, $$\tilde{\vartheta}_c = \mathbb{E}\tilde{\vartheta}_c + \displaystyle\int_0^c\phi(s)\mathrm{d}\mathsf{W}(s).$$
\end{lemma}
A control function can now be introduced for any $\mu>0, \tilde{\vartheta}_c\in L^2(\Omega, \mathbb{H})$ by
\begin{align}
\begin{split}
 %\begin{array}{ll}
    u^\mu(\sigma,\vartheta) = & \displaystyle \mathsf{C}^*\Re^*(c-\sigma)(\mu I + \Delta_0^c)^{-1}\Big[(\mathbb{E}\tilde{\vartheta}_c - \Re(c)h(\vartheta)) + \int_0^c\phi(s)\mathrm{d}\mathsf{W}(s) \\
    & -  \displaystyle\int_0^c\Re(c-s)f(s, \vartheta(s))\mathrm{d}s \\
    & - \displaystyle\int_0^c\Re(c-s)g(s, \vartheta(s))\mathrm{d}\mathsf{W}(s)\Big].
    %\end{array}    
\end{split}
\end{align}
\section{Existence of mild solution of system \eqref{sys}}\label{sec2}
In this section, we establish the existence of a mild solution of the system \eqref{sys}.
We assume that \textbf{(R1)} and \textbf{(R2)} are true in $\mathbb{H}$ for the remainder of this article.\\

The following is the definition of the mild solution for the system \eqref{sys}. 
\begin{definition}
For any given control function $u\in L_\mathcal{F}^2(J, \mathbb{Y})$, a stochastic process $\vartheta$ is said to be a mild solution of \eqref{sys} on $J$ if $\vartheta\in\mathcal{M}(J, L^2(\Omega, \mathbb{H}))$ and satisfies 
    \begin{enumerate}
        \item[(i)] $\vartheta(\sigma), \sigma\in J$ is $\mathcal{F}_\sigma$- adapted and measurable;
        \item[(ii)] $\vartheta(\sigma), \sigma\in J$ can be writing as the following
        $$
        \vartheta(\sigma) = \Re(\sigma)h(\vartheta) + \int_0^\sigma \Re(\sigma-s)[f(s, \vartheta(s)) + Cu(s)]\mathrm{d}s + \int_0^\sigma \Re(\sigma-s)g(s,\vartheta(s))\mathrm{d}\mathsf{W}(s).
        $$
    \end{enumerate}
\end{definition}

\noindent We adopt the following weak assumptions for nonlinear functions $f, g$, and nonlocal function $h$ in order to demonstrate the existence results of mild solution:
\begin{enumerate}
\item[\textbf{(H1)}] The function $f:J\times\mathbb{H}\to\mathbb{H}$ is Carathéodory continuous. Additionally, $\tau_f\in L^1(J,\mathbb{R}^+)$ exists along with a nondecreasing continuous function $\Omega_f:\mathbb{R}^+\to \mathbb{R}^+$ such that
    $$
    \mathbb{E}||f(\sigma,\vartheta)||_\mathbb{H}^2\leq\tau_f(\sigma)\Omega_f(\mathbb{E}||\vartheta||_\mathbb{H}^2), \ a.e\ \sigma\in J, \forall \vartheta\in\mathbb{H}.
    $$
\item[\textbf{(H2)}] The function $g:J\times\mathbb{H}\to L_2^0$ is Carathéodory continuous. Additionally, a function $\tau_g\in L^1(J,\mathbb{R}^+)$ exists, along with a nondecreasing continuous function $\Omega_g:\mathbb{R}^+\to \mathbb{R}^+$ such that
    $$ \mathbb{E}||g(\sigma,\vartheta)||_{L_2^0}^2\leq\tau_g(\sigma)\Omega_g(\mathbb{E}||\vartheta||_\mathbb{H}^2), \ a.e\ \sigma\in J, \forall \vartheta\in\mathbb{H}.
    $$
\item[\textbf{(H3)}] The nonlocal term $h:\mathcal{M}\to\mathbb{H}$ is continuous, and $\zeta:J\times\mathbb{H}\to\mathbb{H}$ is a Carathéodory type function. Furthermore, there exist $\tau_\zeta\in L^1(J,\mathbb{R}^+)$ along with a nondecreasing continuous function $\Omega_\zeta: \mathbb{R}^+\to\mathbb{R}^+$ such that
$$
\mathbb{E}||\zeta(\sigma,\vartheta)||_\mathbb{H}^2\leq\tau_\zeta(\sigma)\Omega_\zeta(\mathbb{E}||\vartheta||_\mathbb{H}^2), \ a.e\ \sigma\in J, \forall \vartheta\in\mathbb{H}.
$$
\end{enumerate}
Let $M=\sup_{\sigma\in [0,c]}\|\Re(\sigma)\|$ and $M_C = \|C\|$.\\
First of all, based on the above assumptions, we then provide a few characteristics of the control $u^\mu$ which is defined above.

\begin{lemma}\label{lemma35}
    If the hypotheses (\textbf{H1})-(\textbf{H3}) are met, then the following conclusions hold for any $\vartheta\in B_r$.
    \begin{enumerate} 
        \item[(i)] $u^\mu(\sigma,\vartheta)$ is continuous in $B_r$;
        \item[(ii)] $\mathbb{E}||u^\mu(\sigma,\vartheta)||^2\leq K_u$,
         where 
        \begin{align*}
        \begin{array}{ll}
        K_u = &\frac{4M_C^2}{\mu^2}M^2\Big(2\mathbb{E}\|\tilde{\vartheta}_c||^2 + 2Tr(Q)\displaystyle\int_0^c\mathbb{E}\|\psi(s)||^2\mathrm{d}s + cM^2\Bigg[\Omega_\zeta(r)||\tau_\zeta||_{L^1([0,c],\mathbb{R}^{+})}\\
        &\\
        &+ \Omega_f(r)||\tau_f\|_{L^1([0,c],\mathbb{R}^{+})} + Tr(Q)c\Omega_\zeta(r)||\tau_\zeta||_{L^1([0,c],\mathbb{R}^{+})}\Bigg]\Big).
        \end{array}
        \end{align*}
%        with $a_0 = cM^2$ and $a_1 = cM^2$.
    \end{enumerate}
\end{lemma}
\begin{proof}
We go through the following steps to prove the two points: (i) and (ii).
    \begin{step}
For $\sigma\in J$ and $\vartheta\in B_r$, by using the H\"older inequality, hypotheses \textbf{(H1)}-\textbf{(H3)} and Lemma \textcolor{blue}{\ref{lemma21}}, we obtain:
     \allowdisplaybreaks
        \begin{align*}
            \mathbb{E}\|u^\mu(\sigma,\vartheta)\|^2 \leq&\displaystyle 4\mathbb{E}\|\mathsf{C}^*\Re^*(c-\sigma)(\mu I + \Delta_0^c)^{-1}\Big(\mathbb{E}\tilde{\vartheta}_c + \int_0^c\phi(s)\mathrm{d}\mathsf{W}(s)\Big)||^2\\ 
            &\\
            & + 4\mathbb{E}||\mathsf{C}^*\Re^*(c-\sigma)(\mu I + \Delta_0^c)^{-1}\Re(c)h(\vartheta)||^2\\
            &\\
             & + \displaystyle 4\mathbb{E}\|\mathsf{C}^*\Re^*(c-\sigma)\int_0^c(\mu I + \Delta_s^c)^{-1}\Re(c-s)f(s, \vartheta(s))\mathrm{d}s||^2\\
             &\\
            & + \displaystyle 4\mathbb{E}||C^*\Re^*(c-\sigma)\int_0^c(\mu I + \Delta_s^c)^{-1}\Re(c-s)g(s, \vartheta(s))\mathrm{d}\mathsf{W}(s)||^2\\
            &\\
            &\leq \frac{4M_C^2}{\mu^2}M^2\Big(2\mathbb{E}||\tilde{\vartheta}_c||^2 + \displaystyle 2Tr(Q)\int_0^c\mathbb{E}||\phi(s)||^2\mathrm{d}s\Big)\\
            &\\
            & + \frac{4M_C^2}{\mu^2}M^4c\Omega_\zeta(r)||\tau_\zeta||_{L[0,c]} + \frac{4M_C^2}{\mu^2}M^4c\Omega_f(r)||\tau_f||_{L^1([0,c],\mathbb{R}^{+})}\\
            &\\
            &+ \frac{4Tr(Q)M_C^2}{\mu^2}M^4c\Omega_g(r)||\tau_g||_{L^1([0,c],\mathbb{R}^{+})}\\
            &\\
            & = \frac{4M_C^2}{\mu^2}M^2\Big(2\mathbb{E}\|\tilde{\vartheta}_c||^2 + 2Tr(Q)\displaystyle\int_0^c\mathbb{E}\|\psi(s)||^2\mathrm{d}s + cM^2\Bigg[\Omega_\zeta(r)||\tau_\zeta||_{L^1([0,c],\mathbb{R}^{+})}\\
        &\\
        &+ \Omega_f(r)||\tau_f\|_{L^1([0,c])} + Tr(Q)\textcolor{blue}{c}\Omega_\zeta(r)||\tau_\zeta||_{L^1([0,c],\mathbb{R}^{+})}\Bigg]\Big)\\
            &\\
            &=K_u,
        \end{align*}
        which implies that \textbf{(ii)} is satisfied.
    \end{step}
   \begin{step}
    Now, we prove that \textbf{(i)} is satisfied. Suppose $\vartheta_n\to \vartheta$ in $B_r$, then we get from \textbf{(H1)-(H3)} 
    $$
    f(\sigma, \vartheta_n(\sigma))\to f(\sigma,\vartheta(\sigma)),\quad g(\sigma, \vartheta_n(\sigma))\to g(\sigma,\vartheta(\sigma)),\quad \zeta(\sigma, \vartheta_n(\sigma))\to \zeta(\sigma, \vartheta(\sigma))\ as\ n\to\infty.
    $$
    Additionally, by using the Lebesgue dominated convergence theorem and the H\"older inequality, for any $\sigma\in J$, we can obtain
    \allowdisplaybreaks
    \begin{align*}
        &\mathbb{E}||\mathsf{C}^*\Re^*(c-\sigma)\int_0^\sigma(\mu I + \Delta_s^c)^{-1}\Re(\sigma-s)[f(s,\vartheta_n(s))-f(s,\vartheta(s))]\mathrm{d}s\|^2\\
        &\leq\frac{M_C^2}{\mu^2}M^4c\int_0^\sigma\mathbb{E}\|f(s, \vartheta_n(s))-f(s, \vartheta(s))||^2\mathrm{d}s\\
        %&\leq\frac{M_C^2}{\mu^2}M^4c\int_0^\sigma\mathbb{E}\|f(s, v_n(s))-f(s, \vartheta(s))\|^2\mathrm{d}s\\
        &\to 0\ (n\to\infty).
    \end{align*}
  On the other hand, using Lemma \ref{lemma21}, the Lebesgue dominated convergence theorem and the H\"older inequality, we arrive at,
    \begin{equation}
    \begin{array}{ll}
      &\displaystyle\mathbb{E}\|\mathsf{C}^*\Re^*(c-\sigma)\int_0^\sigma(\mu I + \Delta_s^c)^{-1}\Re(\sigma-s)[g(s, \vartheta_n(s)) - g(s, \vartheta(s))]\mathrm{d}\mathsf{W}(s)\|^2\\ 
        &\\
        &\displaystyle\leq\frac{Tr(Q)M_C^2M^4}{\mu^2}\int_0^\sigma\mathbb{E}||g(s, \vartheta_n(s)) - g(s, \vartheta(s))||^2\mathrm{d}s\\
        &\\
        &\to 0\ (n\to\infty).
        \end{array}
    \end{equation}
Furthermore, by hypothesis \textbf{(H3)}, we obtain that 
\begin{equation}
    \begin{array}{ll}
      &\displaystyle\mathbb{E}\|\mathsf{C}^*\Re^*(c-\sigma)(\mu I + \Delta_\sigma^c)^{-1}\Re(c)[h(\vartheta_n) - h(\vartheta)]\|^2\\ 
        &\\
        &\displaystyle\leq\frac{M_C^2M^4}{\mu^2}\mathbb{E}||h(\vartheta_n) - h(\vartheta)||^2\mathrm{d}s\\
        &\\
        &\to 0\ (n\to\infty).
        \end{array}
    \end{equation}
    The following outcome follows from the inequality we obtained above:
     \begin{equation*}
         \begin{array}{ll}
         &\displaystyle\mathbb{E}||u^\mu(\sigma, \vartheta_n) - u^\mu(\sigma, \vartheta)||^2\\
         &\\
         \displaystyle\leq &3\mathbb{E}||\mathsf{C}^*\Re^*(c-\sigma)(\mu I + \Delta_s^c)^{-1}\Re(c)[h(\vartheta_n) - h(\vartheta)]||^2\\
         &\\
         & + 3\displaystyle\mathbb{E}||\mathsf{C}^*\Re^*(c-\sigma)\int_0^c(\mu I + \Delta_s^c)^{-1}\Re(c)[f(s, \vartheta_n) - f(s, \vartheta(s))]\mathrm{d}s||^2\\
         &\\
         & + 3\mathbb{E}||C^*\Re^*(c-\sigma)\displaystyle\int_0^c(\mu I + \Delta_s^c)^{-1}\Re(\sigma - s)[g(s, \vartheta_n) - g(s, \vartheta(s)]\mathrm{d}W(s)||^2\\
         &\\
         &\to 0\ (n\to\infty).
         \end{array}
    \end{equation*}
    Therefore, $u^\mu(\sigma, \vartheta)$ is continuous in $B_r$.
   \end{step}  
 This completes the proof.
\end{proof}
We assume that the following hypothesis is true in order to discuss the controllability of the system \eqref{sys}.
\begin{enumerate}
        \item[\textbf{(H4)}] There is a constant $\gamma\in (0,c)$ such that, for all $\sigma\in [\gamma, c]$, 
        \begin{align*}
        \begin{array}{ll}
            & f(\sigma,\vartheta_1(\sigma)) = f(\sigma, \vartheta_2(\sigma)), \\
            & g(\sigma, \vartheta_1(\sigma)) = g(\sigma, \vartheta_2(\sigma)), \\
            & \zeta(\sigma, \vartheta_1(\sigma)) = \zeta(\sigma, \vartheta_2(\sigma)),
        \end{array}
        \end{align*}
        where $\vartheta_1, \vartheta_2\in\mathcal{M}(J, L^2(\Omega, \mathbb{H}))$ with $\vartheta_1(\sigma)=\vartheta_2(\sigma),\sigma\in [\gamma, c].$
    \end{enumerate} 
\begin{theorem}\label{theorem36}
     Suppose that the hypotheses (\textbf{H1})-(\textbf{H3}) are met, and that \textbf{(H4)} is satisfied. Following that, the global issue \eqref{sys} has at least one mild solution in $B_r$ if and only if there is a positive constant $r$ that satisfies the following condition:
    \begin{align}\label{equa28}
    \begin{array}{ll}
       3M^2c\Omega_\zeta(r)||\tau_\zeta||_{L[0,c]}& + 6M^2c\Big(2\Omega_f(r)||\tau_f||_{L[0,c]} + M_C^2cK_u\Big)\\ 
       &\\
       &+ 3Tr(Q)c^{1/2}\Omega_g(r)||\tau_g||_{L^1([0,c])}\leq r.
    \end{array}
    \end{align}
\end{theorem}
\begin{proof}
   Whenever $r>0$, define 
    $$
    B_r(\gamma) = \Big\{v\in\mathcal{M}([\gamma, c], L^2(\Omega,\mathbb{H}));\  \mathbb{E}||\vartheta(\sigma)||^2\leq r, \forall \sigma\in [\gamma, c]\Big\}.
    $$
   It is clear that for each $\vartheta\in B_r(\gamma)$, a function $\bar{\vartheta}\in B_r$ exists such that $\vartheta(\sigma)=\bar{\vartheta}(\sigma),\sigma\in [\gamma, c]$.\\
  %It is  easily seen that for each $\vartheta\in B_r(\gamma)$, there exists a function $\bar{\vartheta}\in B_r$ satisfying $\vartheta(\sigma)=\bar{\vartheta}(\sigma),\sigma\in [\gamma, c]$ is fulfilled
  %obvious that there is a function $\vartheta\in B_r(\gamma)$ so that for any $w\in B_r$, $\vartheta(\sigma)=w(\sigma),\sigma\in [\gamma, c]$ is fulfilled.\\
  Let define the following mappings on $B_r(\gamma)$:
    $$
    (f^*\vartheta)(\sigma) =f(\sigma, \bar{\vartheta}(\sigma)),\ (g^*\vartheta)(\sigma)=g(\sigma,\bar{\vartheta}(\sigma)),\ \sigma\in [0,c],\ \text{and}\ h^*(\vartheta)=h(\bar{\vartheta}).
     $$
     Then, it is simple to verify that $f^*, g^*, h^*$ is well-defined on $B_r(\gamma)$ and continuous by using hypothesis \textbf{(H1)-(H4)}.  Furthermore, assume that the following estimations are met:
\begin{align}\label{equa29}
    \begin{array}{ll}
         &\mathbb{E}||(f^*\vartheta)(\sigma)||^2\leq\tau_f(\sigma)\Omega_f(\mathbb{E}||\vartheta||^2),\ a.e\ \sigma\in J,\ \forall \vartheta\in B_r(\gamma), \\ 
         &\\
         &\mathbb{E}||(g^*\vartheta)(\sigma)||_{L_2^0}^2\leq\tau_g(\sigma)\Omega_g(\mathbb{E}||\vartheta||^2),\ a.e\ \sigma\in J,\ \forall \vartheta\in B_r(\gamma), \\
         &\\
          &\mathbb{E}||(h^*\vartheta)(\sigma)||^2\leq c\Omega_\zeta(r)(||\tau_\zeta||_{L[0,c]}),\ \forall \vartheta\in B_r(\gamma).
         \end{array}
     \end{align}     
     For any $\mu >0$, define an operator $\Psi_\gamma$ on $B_r(\gamma)$ as follows:
     \begin{align*}
         \begin{array}{ll}
              (\Psi_\gamma \vartheta)(\sigma) & = \Re(\sigma)h^*(\vartheta) + \displaystyle\int_0^\sigma \Re(\sigma -s)[(f^*\vartheta)(s) + C\Tilde{u}^\mu(s,\vartheta)]\mathrm{d}s \\
              & + \displaystyle\int_0^\sigma \Re(\sigma-s)(g^*\vartheta)(s)\mathrm{d}\mathsf{W}(s), \sigma\in [\gamma, c],
         \end{array}
     \end{align*}
     where the control $\Tilde{u}^\mu(\sigma,\vartheta)$ is defined by 
\begin{align*}
    \begin{array}{ll}
         \tilde{u}^\mu(\sigma,\vartheta) = & \displaystyle \mathsf{C}^*\Re^*(c-\sigma)(\mu I + \Delta_0^c)^{-1}\Big[\mathbb{E}\tilde{\vartheta}_c - \Re(c)g^*(\vartheta) + \int_0^c\phi(s)\mathrm{d}\mathsf{W}(s) \\
         &\\
         & -  \displaystyle \int_0^c\Re(c-s)f^*(s, \vartheta(s))\mathrm{d}s - \displaystyle \int_0^c\Re(c-s)g^*(\vartheta)(s)\mathrm{d}\mathsf{W}(s)\Big].
    \end{array}
\end{align*}
It's easy to see that $\Tilde{u}^\mu(\sigma,\vartheta)$ satisfies the results of the Lemma \ref{lemma211}.\\
For the rest of the proof, we use Schauder's fixed point Theorem to prove that $\Psi_\gamma$ has a fixed point. We go through four steps to take this:
\begin{step}
   We start by asserting that there is a positive number $r$ such that $\Psi_\gamma$ maps $B_r(\gamma)$ into itself, i.e., there is a positive number $r=r(\gamma)$ for each $\gamma>0$ such that $\Psi_\gamma(B_r)\subset B_r$. For all $\vartheta\in B_r(\gamma)$ and $\sigma\in [\gamma, c]$, it follows from \eqref{equa29} and \eqref{equa28}, Lemma \ref{lemma26} and H\"older's inequality that
   \allowdisplaybreaks
   \begin{align*}
             \mathbb{E}||(\Psi_\gamma \vartheta)(\sigma)||^2&\leq 3\mathbb{E}||\Re(\sigma)h^*(\vartheta)||^2 + 3\mathbb{E}||\displaystyle\int_0^\sigma \Re(\sigma-s)[(f^*\vartheta)(s) + \mathsf{C}\Tilde{u}^\mu(s,\vartheta)]\mathrm{d}s||^2  \\
             &\\
             &+ 3\mathbb{E}||\displaystyle\int_0^\sigma \Re(\sigma-s)(g^*\vartheta)(s)\mathrm{d}\mathsf{W}(s)||^2\\
             &\\
             &\leq 3M^2c\Psi_\zeta(r)||\tau_\zeta||_{L[0,c]} + 3M^2\displaystyle\int_0^\sigma\mathbb{E}||(f^*\vartheta)(s) + \mathsf{C}\Tilde{u}^\mu(s, \vartheta)||^2\mathrm{d}s\\
             &\\
             & + 3Tr(Q)M^2\displaystyle\int_0^\sigma\mathbb{E}||(g^*\vartheta)(s)||_{L_2^0}^2\mathrm{d}s\\
             &\\
             &\leq 3M^2c\Omega_\zeta(r)||\tau_\zeta||_{L[0, c]} + 6M^2c\Bigg(\Omega_f(r)||\tau_f||_{L([0,c])} + M_C^2cK_u\Bigg)\\ \\ & + 3Tr(Q)\Omega_g(r)c^{1/2}||\tau_g||_{L^1([0,1])}\\
             &\\
             &\leq r,
    \end{align*}
which implies that $\Psi_\gamma$ maps $B_r(\gamma)$ into itself.
\end{step}
\begin{step}
We establish that the operator $\Psi_\gamma:B_r(\gamma)\to B_r(\gamma)$ is continuous. Let $v_n, \vartheta\in B_r(\gamma)$ and $\vartheta_n\to \vartheta(n\to\infty)$, then for any $\sigma\in [\gamma, c]$, we have \begin{align*}
        \begin{array}{ll}
             \mathbb{E}||(\Psi_\gamma \vartheta_n)(\sigma) - (\Psi_\gamma \vartheta)(\sigma)||^2&\leq 4\displaystyle\mathbb{E}||\Re(\sigma)(h^*(\vartheta_n) - h^*(\vartheta))||^2\\ 
             &\\
             & + 4\displaystyle\mathbb{E}||\int_0^\sigma \Re(\sigma -s)[(f^*\vartheta_n)(s) - (f^*\vartheta)(s)]\mathrm{d}s||^2\\
             &\\
             & +\displaystyle 4\mathbb{E}||\int_0^\sigma \Re(\sigma-s)[(g^*\vartheta_n)(s) - (g^*\vartheta)(s)]\mathrm{d}\mathsf{W}(s)||^2\\
             &\\
             & + 4\mathbb{E}||\displaystyle\int_0^\sigma C\Re(\sigma -s)[u^\mu(s, \vartheta_n) - u^\mu(s, \vartheta)]\mathrm{d}s||^2.
        \end{array}
    \end{align*}
    Then by the continuity of the nonlocal function $h^*$ and nonlinear functions $f^*, g^*$, and using Lemma \ref{lemma211} and the Lebesgue dominated convergence theorem, we obtain
    $$
     \mathbb{E}||(\Psi_\gamma \vartheta_n)(\sigma) - (\Psi_\gamma \vartheta)(\sigma)||^2\to 0(n\to\infty), \sigma\in [\gamma, c].
    $$
    That implies, $||(\Psi_\gamma \vartheta_n)(\sigma) - (\Psi_\gamma \vartheta)(\sigma)||_\mathcal{M}\to\ 0(n\to\infty)$, mean that $\Psi_\gamma$ is continuous in $B_r(\gamma)$.
\end{step}
\begin{step}\label{step363}
 We show that $\Psi_\gamma$ is a completely continuous operator on $B_r(\gamma)$. Since the compactness of the resolvent operator $\Re(\sigma),\sigma>0$ implies that $\{\Re(\sigma)h^*(\vartheta): \vartheta\in B_r(\gamma)\}$ is pre-compact in $\mathbb{H}$ for each $\sigma\in [\gamma, c]$. So, by using Lemma \ref{lemma26}, we obtain that $\{\Re(\cdot)h^*(\vartheta): \vartheta\in B_r(\gamma)\}$ is equicontinuous.\\
 Additionally, assume $\sigma\in [\gamma, c]$ be fixed, $\forall\alpha\in (0, \sigma)$ and $\forall\gamma >0$ define an operator $\Psi^\alpha$ on $B_r(\gamma)$ using the formula:\\
 \allowdisplaybreaks
  \begin{align*}
      \begin{array}{ll}
           (\Psi^\alpha \vartheta)(\sigma) & = \displaystyle\int_0^{\sigma -\alpha} \Re(\sigma -s)[(f^*\vartheta)(s) + \mathsf{C}\Tilde{u}^\mu(s,\vartheta)]\mathrm{d}s 
            + \displaystyle\int_0^{\sigma -\alpha}\Re(\sigma -s)(g^*\vartheta)(s)\mathrm{d}\mathsf{W}(s)\\
           &\\
           &=\displaystyle \Re(\alpha)\int_0^{\sigma -\alpha}\Re((\sigma -s) - \alpha )[(f^*\vartheta)(s) + \mathsf{C}\Tilde{u}^\mu(s,\vartheta)]\mathrm{d}s  \\
           &\\
           & + \displaystyle \Re(\alpha)\int_0^{\sigma -\alpha}\Re((\sigma -s) - \alpha)(g^*\vartheta)(s)\mathrm{d}\mathsf{W}(s).
      \end{array}
  \end{align*}
  Then, by the compactness of $\Re(\alpha)$, the set $\{(\Psi^\alpha \vartheta)(\sigma): \vartheta\in B_r(\gamma)\}$ is relatively compact in $\mathbb{H}$. We denote
  \begin{align*}
      \begin{array}{ll}
           (\Psi_1^\alpha \vartheta)(\sigma) =\displaystyle\int_0^{\sigma -\alpha} \Re(\sigma -s)[(f^*\vartheta)(s) + \mathsf{C}\Tilde{u}^\mu(s,\vartheta)]\mathrm{d}s  + \displaystyle\int_0^{\sigma -\alpha}\Re(\sigma -s)(g^*\vartheta)(s)\mathrm{d}\mathsf{W}(s),
      \end{array}
  \end{align*}
  for any $\vartheta\in B_r(\gamma)$. Using \eqref{equa29}, the Lemmas \ref{lemma21} and \ref{lemma28}, and H\"older inequality, we obtain
 \allowdisplaybreaks
\begin{align*}
           &\mathbb{E}||(\Psi_1^\alpha \vartheta)(\sigma)  - (\Psi^\alpha \vartheta)(\sigma)||^2\\
           &\\
           &= \displaystyle \mathbb{E}||\int_0^{\sigma -\alpha}[\Re(\alpha)\Re((\sigma -s) -\alpha) - \Re(\sigma -s)]\Big((f^*\vartheta)(s) + \mathsf{C}\Tilde{u}^\mu(s,\vartheta)\Big)\mathrm{d}s\\
           &\\
            & + \displaystyle\int_0^{\sigma -\alpha}[\Re(\alpha)\Re((\sigma -s) -\alpha) - \Re(\sigma-s)](g^*\vartheta)(s)\mathrm{d}\mathsf{W}(s)||^2\\
            &\\
           &\leq 4\mathbb{E}||\displaystyle\int_0^\sigma[\Re(\alpha)\Re((\sigma -s) -\alpha) -\Re(\sigma -s)]\Big((f^*\vartheta)(s) + \mathsf{C}\Tilde{u}^\mu(s,\vartheta)\Big)\mathrm{d}s||^2 \\
           &\\
           & + 4\mathbb{E}||\displaystyle\int_{\sigma-\alpha}^\sigma[\Re(\alpha)\Re((\sigma-s) -\alpha) -\Re(\sigma-s)]\Big((f^*\vartheta)(s) + \mathsf{C}\Tilde{u}^\mu(s,\vartheta)\Big)\mathrm{d}s||^2 \\
           &\\
           & + \mathbb{E}||\displaystyle\int_0^\sigma[\Re(\alpha)\Re((\sigma -s) -\alpha) - \Re(\sigma -s)](g^*\vartheta)(s)\mathrm{d}\mathsf{W}(s)||^2 \\
           &\\
           &+ \mathbb{E}||\displaystyle\int_{\sigma - \alpha}^\sigma[\Re(\alpha)\Re((\sigma -s) -\alpha) - \Re(\sigma -s)](g^*\vartheta)(s)\mathrm{d}\mathsf{W}(s)||^2\\
           &\\
            & \leq 4(\gamma h)^2\displaystyle\int_0^\sigma\mathbb{E}\|(f^*\vartheta)(s) + \mathsf{C}\Tilde{u}^\mu(s,\vartheta)\|^2\mathrm{d}s + 4(\gamma h)^2\displaystyle\int_{\sigma -\alpha}^\sigma\mathbb{E}\|(f^*\vartheta)(s) + \mathsf{C}\Tilde{u}^\mu(s,\vartheta)\|^2\mathrm{d}s\\
           &\\
           & + 4Tr(Q)(\gamma h)^2\displaystyle\int_0^\sigma\mathbb{E}||(g^*\vartheta)(s)\|^2\mathrm{d}s  + 4Tr(Q)(\gamma h)^2\displaystyle\int_{\sigma -\alpha}^\sigma\mathbb{E}||(g^*\vartheta)(s)\|^2\mathrm{d}s 
           \\
           &\\
           & \leq 8(\gamma h)^2c\Big(\Omega_f(r)||\tau_f||_{L^1[0,c]} + M_C^2cK_u\Big) + 8(\gamma h)^2c\Big(2\Omega_f(r)||\zeta_f||_{L^1[0, c]}) + 2 M_C^2cK_u\Big)\\
           &\\
           & + 4(\gamma h)^2c^{1/2}Tr(Q)\Omega_g(r)||\tau_g|| + 4(\gamma h)^2c^{1/2}Tr(Q)\Omega_g(r)||\tau_g||\\
           &\\
           &\to 0(\gamma\to 0).
 \end{align*}
  As a result, for each $\sigma\in [\gamma, c]$, there exists relatively compact set arbitrarily close to the set $\{(\Psi_1v)(\sigma): \vartheta\in B_r(\gamma)\}$ in $\mathbb{H}$.
  Hence, $\{(\Psi_1 \vartheta)(\sigma): \vartheta\in B_r(\gamma)\}$ is also relatively compact in $\mathbb{H}$ for $\sigma\in [\gamma, c]$.\\
  
\noindent For the second part of this step, we will show that $\Psi_1(B_r(\gamma))$ is an equicontinuous family of functions on $[\gamma, c]$. For any $\vartheta\in B_r(\gamma)$ and $\gamma<\sigma_1<\sigma_2\leq c$, we obtain 
  \allowdisplaybreaks
  \begin{align*}
           \mathbb{E}||(\Psi_1 \vartheta)(\sigma_2)\ \ - &(\Psi_1\vartheta)(\sigma_1)||^2  = \displaystyle 6\mathbb{E}||\int_{\sigma_1}^{\sigma_2}\Re(\sigma_2-s)[(f^*\vartheta)(s) + \mathsf{C}\Tilde{u}^\mu(s,\vartheta)]\mathrm{d}s||^2\\
           &\\
           & + \displaystyle 6\mathbb{E}\|\int_0^{\sigma_1}[\Re(\sigma_2-s) - \Re(\sigma_1 - s)][(f^*\vartheta)(s) + \mathsf{C}\Tilde{u}^\mu(s,\vartheta)]\mathrm{d}s\|^2\\
           &\\
           & + \displaystyle 6\mathbb{E}||\int_{\sigma_1}^{\sigma_2}\Re(\sigma_2-s)(g^*\vartheta)(s)\mathrm{d}\mathsf{W}(s)||^2\\
           &\\
           & + \displaystyle 6\mathbb{E}\|\int_0^{\sigma_1}[\Re(\sigma_2-s) - \Re(\sigma_1-s)](g^*\vartheta)(s)\mathrm{d}\mathsf{W}(s)\|^2 \\
           &\\
           &=l_1 + l_2 + l_3 + l_4.
 \end{align*}
To demonstrate that $ \mathbb{E}||(\Psi_1 \vartheta)(\sigma_2) - (\Psi_1 \vartheta)(\sigma_1)||^2\to 0(\sigma_2\to \sigma_1),
  $
 all that is left to do is verify that $l_i\to 0$ independently of $\vartheta\in B_r(\gamma)$ when $\sigma_2\to \sigma_1$ for $i=1, 2,\cdots, 4$.\\
  For $l_1$ and $l_2$ when, using the assumption \eqref{equa29}, Lemma \ref{lemma21}, Lemma \ref{lemma26} and H\"older inequality, we obtain the following estimations,
 \allowdisplaybreaks
 \begin{align*}
         l_1 &= \displaystyle 6\mathbb{E}\|\int_{\sigma_1}^{\sigma_2}\Re(\sigma_2-s)[(f^*\vartheta)(s) + \mathsf{C}\Tilde{u}^\mu(s,\vartheta)]\mathrm{d}s\|^2\\
         &\\
         &\leq 6M^2\displaystyle\int_{\sigma_1}^{\sigma_2}\mathbb{E}\|[(f^*\vartheta)(s) + \mathsf{C}\Tilde{u}^\mu(s,\vartheta)]\|^2\mathrm{d}s\\
         &\\
         &\leq 12M^2\Big(\Omega_f(r)||\zeta_f||_{L^1[0,c]} + M_C^2(\sigma_2-\sigma_1)K_u\Big)(\sigma_2 - \sigma_1)\\
         &\\
         &\to 0(\sigma_2\to \sigma_1).\\
         &\\
         l_3 &= \displaystyle 6\mathbb{E}\|\int_{\sigma_1}^{\sigma_2}\Re(\sigma_2-s)(g^*\vartheta)(s)\mathrm{d}\mathsf{W}(s)\|^2\\
         &\\
         &\leq 6M^2Tr(Q)\displaystyle\int_{\sigma_1}^{\sigma_2}\mathbb{E}||(g^*\vartheta)(s)||^2\mathrm{d}s\\
         &\\
         &\leq 6M^2 Tr(Q)\Omega_g(r)||\tau_g||_{L^1[0,1]} (\sigma_2-\sigma_1)^{1/2}\\
         &\\
         &\to 0(\sigma_2\to \sigma_1).
 \end{align*}
  Additionally, if $0<\gamma<\sigma_1$ is small enough, we derive the following inequality for $I_2$ and $I_4$:
 \allowdisplaybreaks
 \begin{align*}
         l_2 &= \displaystyle 6\mathbb{E}\|\int_0^{\sigma_1}[\Re(\sigma_2-s) - \Re(\sigma_1-s)]\Big[(f^*\vartheta)(s) + \mathsf{C}\Tilde{u}^\mu(s,\vartheta)\Big]\mathrm{d}s\|^2 \\
         &\\
         &\leq\displaystyle 12\mathbb{E}\|\int_0^{\sigma_1 - \alpha}[\Re(\sigma_2-s) - \Re(\sigma_1-s)]\Big[(f^*\vartheta)(s) + \mathsf{C}\Tilde{u}^\mu(s,\vartheta)\Big]\mathrm{d}s\|^2\\
         &\\
         & + \displaystyle 12\mathbb{E}\|\int_{\sigma_1-\alpha}^{\sigma_1}[\Re(\sigma_2-s) - \Re(\sigma_1-s)]\Big[(f^*\vartheta)(s) + \mathsf{C}\Tilde{u}^\mu(s,\vartheta)\Big]\mathrm{d}s\|^2\\
         &\\
         &\leq 12\sup_{s\in[0, \sigma_1 -\alpha]}||\Re(\sigma_2 - s) - \Re(\sigma_1 - s)||^2\Big(2\Omega_f(r)\|\tau_f\|_{L[0, c]} + 2M_C^2K_u \Big)(\sigma_1 - \alpha)\\
         &\\
         & + 12\sup_{s\in[\sigma_1 -\alpha, \sigma_1]}||\Re(\sigma_2 - s) - \Re(\sigma_1 - s)||^2\Big(2\Omega_f(r)||\tau_f||_{L[0,c]} + 2M_C^2K_u\Big)\alpha\\
         &\\
         &\to 0(\sigma_2\to \sigma_1).\\
         &\\
         l_4 &= \displaystyle 6\mathbb{E}\|\int_0^{\sigma_1}[\Re(\sigma_2-s) - \Re(\sigma_1-s)](g^*\vartheta)(s)\mathrm{d}\mathsf{W}(s)\|^2 \\
         &\\
         &\displaystyle\leq 12\mathbb{E}\|\int_0^{\sigma_1 -\alpha}[\Re(\sigma_2-s) - \Re(\sigma_1-s)](g^*\vartheta)(s)\mathrm{d}\mathsf{W}(s)\|^2\\
         &\\
         & + 12\mathbb{E}\|\int_{\sigma_1-\alpha}^{\sigma_1}[\Re(\sigma_2-s) - \Re(\sigma_1 - s)](g^*\vartheta)(s)\mathrm{d}\mathsf{W}(s)\|^2\\
         &\\
         &\leq 12Tr(Q)\Omega(r)\sup_{s\in [0, \sigma_1 -\alpha]}||\Re(\sigma_2-s) - \Re(\sigma_1 - s)||^2||\tau_g||_{L^1([0,1])}\times\Big(\sigma_1 - \alpha \Big)^{1/2}\\
         &\\
         & + 12Tr(Q)\sup_{s\in[\sigma_1 -\alpha, \sigma_1]}||\Re(\sigma_2 - s) - \Re(\sigma_1 - s)||^2\Omega(r)||\tau_g||_{L^1([0,1])}\alpha^{1/2}\to 0\\
         &\\
         &\to 0(\sigma_2\to \sigma_1).
 \end{align*}
  Above all, we obtain that $l_{i=1,2,3,4} \to 0$ as $\sigma_2\to \sigma_1$, and $\gamma\to 0$, which means $\Psi_1(B_r(\gamma))$ is equicontinuous.\\
  
  As a result, the precompactness of\, $\Psi_\gamma(B_r(\gamma))$ is demonstrated by the \textbf{Arzela-Ascoli} Theorem. Thus, by utilizing \textbf{Schauder's} fixed point Theorem, we conclude that $\Psi_\gamma$ has at least one fixed point $\Tilde{v}\in B_r(\gamma)$, i.e:
    \begin{align*}
      \begin{array}{ll}
        \vartheta(\sigma) &= \Re(\sigma)h^*(\hat{\vartheta}) + \displaystyle\int_0^{\sigma}\Re(\sigma-s)\Big[f^*(\hat{\vartheta})(s) + C\Tilde{u}^\mu(s, \hat{\vartheta})\Big]\mathrm{d}s\\
         & + \displaystyle\int_0^{\sigma}\Re(\sigma-s)g^*(s,\hat{\vartheta})(s)\mathrm{d}\mathsf{W}(s), \sigma\in [0,c].
      \end{array}
 \end{align*}
 Set 
   \begin{align*}
         \Tilde{z}(\sigma) &= \Re(\sigma)h^*(\hat{\vartheta}) + \displaystyle\int_0^{\sigma}\Re(\sigma -s)\Big[f^*(\hat{\vartheta})(s) + C\Tilde{u}^\mu(s, \Tilde{z})\Big]\mathrm{d}s\\
         & + \displaystyle\int_0^{\sigma}\Re(\sigma -s)g^*(s,\hat{\vartheta})(s)\mathrm{d}\mathsf{W}(s), \sigma\in [0,c].
 \end{align*}
 Obviously, $\hat{\vartheta}(\sigma)=\Tilde{z}(\sigma)$ for $\sigma\in [\gamma, c]$. It is clear from the definitions of the functions $f^*, g^*$, and $h^*$ that,
  \begin{align*}
         \Tilde{z}(\sigma) &= \Re(\sigma)h^*(\Tilde{z}) + \displaystyle\int_0^{\sigma}\Re(\sigma -s)\Big[f^*(\Tilde{z})(s) + \mathsf{C}\Tilde{u}^\mu(s, \Tilde{z})\Big]\mathrm{d}s\\
         & + \displaystyle\int_0^{\sigma}\Re(\sigma-s)g^*(s,\Tilde{z})(s)\mathrm{d}\mathsf{W}(s), \sigma\in [0, c],
 \end{align*}
 which implies, $\Tilde{z}$ is a mild solution of system \eqref{sys} in $B_r$.
\end{step}
This completes the proof.
\end{proof}
For all $\gamma\in (0, c)$ and an arbitrary $\vartheta\in\mathcal{M}(J, L^2(\Omega, \mathbb{H}))$, define 
\begin{equation}\label{equa32}
 (\mathcal{N}_\gamma \vartheta)(\sigma) =\begin{cases}
	 \vartheta(\gamma), \sigma\in [0, \gamma],\\
		\vartheta(\sigma), \sigma\in [\gamma, c],
	\end{cases}
\end{equation}
and 
\begin{align*}
    f_\gamma(\sigma, \vartheta(\sigma)) = f(\sigma, (\mathcal{N}\vartheta)(\sigma)), \ \sigma\in [\gamma, c]\\
    \\
    g_\gamma(\sigma,\vartheta(\sigma)) = g(\sigma, (\mathcal{N}\vartheta)(\sigma)), \ \sigma\in [\gamma, c]\\
    \\
    \zeta_\gamma(\sigma,\vartheta(\sigma)) = \zeta(\sigma, (\mathcal{N}\vartheta)(\sigma)), \ \sigma\in [\gamma, c]
\end{align*}
It is clear that the functions $f_\gamma, h_\gamma$ and $\zeta_\gamma$ defined above meet the constraint \textbf{(H4)}, resulting in the following lemma:
\begin{lemma}\label{lema213}
    Assume that \textbf{(R1)-(R2)} are satisfied. If the assumptions \textbf{(H1)-(H3)} hold, then the following nonlocal problem 
    \begin{equation}\label{equa31}
		\begin{cases}
			\vartheta^\prime(\sigma) = A\vartheta(\sigma) + \displaystyle\int_0^\sigma\Pi(\sigma-s) \vartheta(s)ds + g_\gamma(\sigma,\vartheta(\sigma))\mathrm{d}\mathsf{W}(\sigma) + Cu^\mu(\sigma, (\mathcal{N}_{\gamma_n}\vartheta)(\sigma)),\\
			\vartheta(0) = \displaystyle\int_0^\sigma\zeta_\gamma(s, \vartheta(s))\mathrm{d}s,
		\end{cases}
	\end{equation}
    has at least one mild solution in $B_r$, given that there is a positive constant $r$ such that \eqref{equa28} is satisfied.
\end{lemma}
\begin{proof}
    The proof is similar to the proof of Theorem \ref{theorem36} above. 
\end{proof}

\section{Approximate controllability results}\label{sec3}
In this section, we analyze the approximate controllability of the stochastic dynamical control system $\eqref{sys}$ by utilizing approximation techniques and a diagonal argument.
\begin{theorem}\label{theo31}
  Suppose that \textbf{(H1)}-\textbf{(H3)} are satisfied, then the stochastic control systems with nonlocal initial conditions \eqref{sys} have at least one mild solution in $\mathcal{M}(J, L^2(\Omega, \mathbb{H})),$ provided that there exists a positive constant $r$ such that \eqref{equa28} is satisfied.
\end{theorem}
\begin{proof}
Let us consider $\{\gamma_n: n\in\mathbb{N}\}$ to be a decreasing sequence in $(0, c)$ with $\lim_{n\to\infty}\gamma_n = 0$. For any $n$, by Lemma \ref{lema213}, we assume that the following system 
	\begin{equation}
		\begin{cases}
			\vartheta(\sigma) = A\vartheta(\sigma) + \displaystyle\int_0^\sigma\Pi(\sigma -s)\vartheta(s)\mathrm{d}s + f_{\gamma_n}(\sigma, \vartheta(\sigma)) + g_{\gamma_n}(\sigma,\vartheta(\sigma))\mathrm{d}\mathsf{W}(\sigma) + Cu^\mu(\sigma, (\mathcal{N}_{\gamma_n}\vartheta)(\sigma)),\\
			\vartheta(0) = \displaystyle\int_0^\sigma\zeta_{\gamma_n}(s, \vartheta(s))\mathrm{d}s,
		\end{cases}
	\end{equation}
admits at least one mild solution $\vartheta_n\in B_r$ for the constant $r$ satisfying \eqref{equa28}, which is defined by
\begin{align*}
        \vartheta_n(\sigma) = & \displaystyle \Re(\sigma)\int_0^\sigma\zeta_{\gamma_n}(s, \vartheta(s))\mathrm{d}s + \int_0^\sigma \Re(\sigma -s)\Big[f_{\gamma_n}(s,\vartheta_n(s))\\
        & + \displaystyle Cu^\mu(s,\mathcal{N}_{\gamma_n}\vartheta_n(s))\Big]\mathrm{d}s + \int_0^\sigma \Re(\sigma -s)(g_{\gamma_n}(s, \vartheta_n(s)))\mathrm{d}\mathsf{W}(s),\ \sigma\in [0,c].
\end{align*}
Let's suppose the following: 
	\begin{equation}
		x_n(\sigma) =\begin{cases}
		 \vartheta_n(\gamma_n), \sigma\in [0, \gamma_n],\\
			\vartheta_n(\sigma), \sigma\in [\gamma_n, c],
		\end{cases}
	\end{equation}
	then $x_n\in B_r$. We may deduce that given the characteristics of $f_{\gamma_n}, g_{\gamma_n}$ and $\zeta_{\gamma_n}$,
	\begin{align}\label{equa33}
    \begin{array}{ll}
        \vartheta_n(\sigma) = & \displaystyle \Re(\sigma)\int_0^\sigma \zeta(s, x_n)\mathrm{d}s + \int_0^\sigma \Re(\sigma -s)\Big[f(s,x_n(s))\\
        & + \displaystyle \mathsf{C}u^\mu(s, x_n(s))\Big]\mathrm{d}s + \int_0^\sigma \Re(\sigma -s)(g(s, x_n(s)))\mathrm{d}\mathsf{W}(s),\ \sigma\in [0,c].
    \end{array}
\end{align}
We shall then demonstrate that the set $\{\vartheta_n : n\in\mathbb{N}\}$ is precompact in $\mathcal{M}(J, L^2(\Omega,\mathbb{H}))$. In light of this, we provide the following notations:
\allowdisplaybreaks
\begin{align*}
    \vartheta_n(\sigma) &= \Re(\sigma)\int_0^\sigma\zeta(s, \vartheta_n(s))\mathrm{d}s,\ \sigma\in [0, c],\\
    &\\
    \psi_n(\sigma) = &\int_0^\sigma \Re(\sigma -s)\Big[(f(s,\vartheta_n(s)) + \displaystyle \mathsf{C}u^\mu(s,\vartheta_n(s))\Big]\mathrm{d}s + \int_0^\sigma \Re(\sigma - s)g(s, \vartheta_n(s))\mathrm{d}\mathsf{W}(s),\ \sigma\in [0,c].
\end{align*}
Consequently,  all that is left to do is to demonstrate that the sets $\{\vartheta_n:n\in\mathbb{N}\}$ and $\{\psi_n: n\in\mathbb{N}\}$ are precompact in the space $\mathcal{M}(J, L^2(\Omega, \mathbb{H}))$. \\

\noindent To do so, we go through the following steps:
\begin{step}
First, we can see that $x_n\in B_r$ from the definition of $x_n(\sigma)$. So, the hypotheses \textbf{(H1)-(H3)} are satisfied for functions $f(s, x_n),\; g(s, x_n(s))$ and $h(s, x_n(s))$. Furthermore, the control function $u^\mu(s, x_n)$ satisfies the properties \textbf{(i)} and \textbf{(ii)} of Lemma \ref{lemma35}. So, we easily prove that the set $\{\psi_n:n\in\mathbb{N}\}$ is precompact in $\mathcal{M}(J, L^2(\Omega, \mathbb{H}))$ by using the same arguments as in Theorem \ref{theorem36}.
\end{step}

\begin{step}
Secondly, we'll show that the set $\{\vartheta_n : n\in\mathbb{N}\}$ is precompact in $\mathcal{M}(J, L^2(\Omega,\mathbb{H}))$.\\ It turns out that all we need to show is that the set $$\{\int_0^\sigma \zeta(s, x_n(s))\mathrm{d}s:n\in\mathbb{N}\}$$ is precompact in $\mathbb{H}$.\\
Let $\{\mu_n:n\in\mathcal{N}\}$ be a decreasing sequence in $(0,c)$ such that $\lim_{n\to\infty}\mu_n=0$.
 \begin{itemize}
     \item For all $n\in\mathbb{N}$ and $\sigma\in[\mu_1, c]$, define the function $z_n : [\mu_1, c]\to\mathbb{H}$ by  $z_n(\sigma)=\vartheta_n(\sigma)$. From the fact that $x_n\in B_r$, the set $$\{\int_0^\sigma \zeta(s, x_n(s))\mathrm{d}s:n\in\mathbb{N}\}$$ is bounded.\\ 
     Meanwhile, from the fact that the resolvent operator $\Re(\sigma)$ is compact, and hence norm continuous (see, Lemma \ref{lemma27}) for $\sigma>0$, which  implies that the set $$\displaystyle\{\Re(\sigma)\int_0^\sigma \zeta(s, x_n(s))\mathrm{d}s:n\in\mathbb{N}\}$$ 
is precompact in $\mathbb{H}$ for any $\sigma\in [\mu_1, c]$ and $\displaystyle\{\Re(\cdot)\int_0^\sigma \zeta(s, x_n(s))\mathrm{d}s:n\in\mathbb{N}\}$ is equicontinuous.\\ Consequently, we conclude based on \textbf{Arzela-Ascoli's}  Theorem, that
  $$\{\Re(\sigma)\int_0^\sigma\zeta(s, x_n(s))\mathrm{d}s: n\in\mathbb{N}\}$$ is precompact in $\mathcal{M}([\mu_1, c], L^2(\Omega, \mathbb{H}))$.
Combining this with the precompact nature of $\{\psi_n:n\in\mathbb{N}\}$
in $\mathcal{M}(J, L^2(\Omega, \mathbb{H}))$, we obtain that $\{z_n:n\in\mathbb{N}\}$ is precompact in $\mathcal{M}([\mu_1, c], L^2(\Omega, \mathbb{H}))$. Thus, we can find a sub-sequence $\{\vartheta_n^1:n\in\mathbb{N}\}\subset\{\vartheta_n:n\in\mathbb{N}\}$ which is a \textbf{Cauchy} sequence in $\mathcal{M}([\mu_1, c], L^2(\Omega, \mathbb{H}))$.
\item In the same manner, we can construct a sub-sequence $\{\vartheta_n^2:n\in\mathbb{N}\}\subset\{\vartheta_n^1:n\in\mathbb{N}\}$, which is a Cauchy sequence in $\mathcal{M}([\mu_2, c], L^2(\Omega, \mathbb{H}))$. We demonstrate that there exists a sub-sequence $\{\vartheta_n^*:n\in\mathbb{N}\}\subset \{\vartheta_n:n\in\mathbb{N}\}$ which is a \textbf{Cauchy} sequence in $\mathcal{M}([\mu_n, c], L^2(\Omega, \mathbb{H}))$ by following the previous steps again while employing a diagonal argument.
\item In addition, $\{\vartheta_n^*(\sigma): n\in\mathbb{N}\}$ is a \textbf{Cauchy} sequence in $\mathbb{H}$ for every $\sigma\in (0, c]$.
 \end{itemize}
Consequently, there exists a continuous function $\vartheta^*:(0, c]\to L^2(\Omega, \mathbb{R})$ such that for any $\mu_k$,
\begin{align}\label{equa44}
    \lim_{n\to\infty}\max_{\sigma\in [\nu, c]}\mathbb{E}||\vartheta_n^*(\sigma) - \vartheta^*(\sigma)||^2=0.
\end{align}
In addition, we demonstrate that $\{h(\vartheta_n^*): n\in\mathbb{N}\}$ is a Cauchy sequence in $\mathbb{H}$. Let $\gamma\in (0,c)$, thus for any $\vartheta_1, \vartheta_2\in\mathcal{M}(J, L^2(\Omega, \mathbb{H}))$ with $\vartheta_1(\sigma)=\vartheta_2(\sigma), \sigma\in [\gamma, c]$ we obtain
\begin{align}
\begin{split}
     \mathbb{E}||h(\vartheta_1) - h(\vartheta_2)||^2 &= \mathbb{E}||\displaystyle\int_0^\gamma[\zeta(s,\vartheta_1(s)) - \zeta(s, \vartheta_2(s))]\mathrm{d}s||^2\\
       &\le\int_0^\gamma\mathbb{E}\|\zeta(s,\vartheta_1(s)) - \zeta(s, \vartheta_2(s))\|^2\mathrm{d}s\\
       &\to 0\text{ for }\gamma\to 0.
\end{split}
\end{align}
Thus, it is easy to see that $\forall \epsilon>0$, there is a positive constant $\gamma_0<c$ such that 
\begin{align}\label{equa34}
 \mathbb{E}||h(\vartheta_1) - h(\vartheta_2)||^2<\frac{\epsilon}{16}
\end{align}
for any $\vartheta_1,\vartheta_2\in\mathcal{M}(J, L^2(\Omega, \mathbb{H}))$ with $\vartheta_1(\sigma)=\vartheta_2(\sigma),\sigma\in [\gamma_0, c]$. Let's define the function $z(\sigma)$ by
\begin{equation}\label{equa35}
		z(\sigma) =\begin{cases}
		 \vartheta^*(\gamma_0), \sigma\in [0, \gamma_0],\\
			\vartheta^*(\sigma),\sigma\in [\gamma_0, c].
		\end{cases}
	\end{equation}
Obviously, $z\in\mathcal{M}(J, L^2(\Omega, \mathbb{H}))$. By the limit \eqref{equa44}, we have $$
  \lim_{n\to\infty}\max_{\sigma\in[\gamma_0, c]}\mathbb{E}||\vartheta_n^*(\sigma) - z^*(\sigma)||^2=0.
$$
From the definition of $\mathcal{N}_\gamma$, we can easily see that $$
  \lim_{n\to\infty}||\mathcal{N}_{\gamma_0} \vartheta_n^* - z||_\mathcal{M}=0.
$$ 
By using continuity of $h$, we can locate a natural integer $k$, such that
$$
  \lim_{n\to\infty}\mathbb{E}||h(\mathcal{N}_{\gamma_0} \vartheta_n^*) - h(z)||^2<\frac{\epsilon}{16}, n>k.
$$ 
As a result, we have for any $m, n>k$,
\begin{align*}
    \begin{array}{ll}
          \mathbb{E}||h(\vartheta_m^*) - h(\vartheta_n^*)||^2 & \leq 4\mathbb{E}||h(\vartheta_m^*) - h(\mathcal{N}_{\gamma_0}\vartheta_n^*)||^2 + 4\mathbb{E}||h(\mathcal{N}_{\gamma_0}\vartheta_m^*) - h(z)||^2\\
          &\\
         & + 4\mathbb{E}|||h(\mathcal{N}_{\gamma_0}\vartheta_n^*) - h(z)|| + 4\mathbb{E}|||h(\vartheta_n^*) - h(\mathcal{N}_{\gamma_0}\vartheta_n^*)||^2\\
         &\\
         &< \epsilon.
    \end{array}
\end{align*}
The above inequality implies that $\{h(\vartheta_n^*):n\in\mathbb{N}\}$ is a Cauchy sequence in the Hilbert space $\mathbb{H}$, namely, $\{h(\vartheta_n):n\in\mathbb{N}\}$ is precompact in $\mathbb{H}$. Finally, we can quickly confirm that using the notation \eqref{equa32},
$$
\mathbb{E}||h(\vartheta_n)-h(z_n)||^2\underset{n\to\infty}{\to 0},
$$
which implies that $$\{\int_0^c\zeta(s,z_n(s))\mathrm{d}s:n\in\mathbb{N}\}$$ is precompact in $\mathbb{H}$.
\end{step}
\noindent Above, we have established that the set $\{\vartheta_n:n\in\mathbb{N}\}$ is precompact in $\mathcal{M}(J, L^2(\Omega, \mathbb{H}))$. Consequently, a subsequence of $\{\vartheta_n:n\in\mathbb{N}\}$ still denoted by $\{\vartheta_n:n\in\mathbb{N}\}$ exists, and a function $\vartheta_0\in B_r$ such that
\begin{align}\label{equa36}
    \lim_{n\to\infty}||\vartheta_n - \vartheta_0||_\mathcal{M}^2=0.
\end{align}
According to the definition of $x_n$, we obtain that 
\begin{align}\label{equa37}
\begin{array}{ll}
   ||x_n - \vartheta_0||_\mathcal{M}^2 &= \max_{\sigma\in J}\mathbb{E}||x_n(\sigma) - \vartheta_0(\sigma)||^2\\
   &\\
   &\leq \max_{\sigma\in [0, \gamma_n]}\mathbb{E}||\vartheta_n(\gamma_n) - \vartheta_0(\sigma)||^2 + \max_{\sigma\in [\gamma_n, c]}\mathbb{E}||\vartheta_n(\gamma_n) - \vartheta_0(\sigma)||^2
   \\
   &\\
   &\leq 2\mathbb{E}||\vartheta_n(\gamma_n) - \vartheta_0(\gamma_n)||^2 + 2\max_{\sigma\in [0, \gamma_n]}\mathbb{E}||\vartheta_0(\gamma_n) - \vartheta_0(\sigma)||^2 + ||\vartheta_n - \vartheta_0||_\mathcal{M}^2\\
   &\\
   &\leq 3||\vartheta_n - \vartheta_0||_\mathcal{M}^2 + 2\max_{\sigma\in[0,\gamma_n]}\mathbb{E}||v_0(\gamma_n)-\vartheta_0(\sigma)||^2.
\end{array}
\end{align}
By using inequality \eqref{equa36} and \eqref{equa37} and taking the limit of \eqref{equa33} as $n\to\infty$ one obtain that 
\begin{align*}
    \vartheta_0(\sigma) &= \Re(\sigma)\int_0^\sigma\zeta(s, \vartheta_0(s))\mathrm{d}s + \int_0^\sigma \Re(\sigma -s)[f(s, \vartheta_0(s)) + Cu^\mu(s,\vartheta_0)]\mathrm{d}s \\ & +
\int_0^\sigma \Re(\sigma-s)(g(s, \vartheta_0(s))\mathrm{d}\mathsf{W}(s),\ \sigma\in[0,c]
\end{align*}
This implies that $\vartheta_0\in\mathcal{M}(J, L^2(\Omega, \mathbb{H}))$ is a mild solution of system \eqref{sys}. Finally, the proof of Theorem \ref{theo31} is finished.
\end{proof}
%\begin{remark}
  % Without assuming compactness or Lipschitz continuity of the nonlocal function, we used a new method to establish the existence result of the nonlocal system \eqref{sys}. Our findings, therefore, have wider applications.  
%\end{remark}
We use the following additional hypotheses in order to prove the approximate controllability of system \eqref{sys}.
\begin{enumerate}
    \item[\textbf{(H5)}] The nonlinear function $f:J\times\mathbb{H}\to\mathbb{H}$ and $g:J\times\mathbb{H}\to L_2^0$ are uniformly bounded.
    \item[\textbf{(H6)}] The linear integrodifferential system \eqref{linsys} is approximately controllable on the interval $[0, c]$.
\end{enumerate}
\begin{remark}\label{rem41}
In light of \cite{mahmudov2003controllability}, the assumption \textbf{(H6)} is equivalent to $\mu S(\mu, \Delta_0^c)\to 0$ as $\mu\to 0^+$ in the strong operator topology.
\end{remark}
\begin{theorem}\label{theo42}
    Assume that \textbf{(H5)} and (\textbf{(H6)} are satisfied, in addition to the Theorem \ref{theo31} presumptions holding true. So, the nonlocal system \eqref{sys} is approximately controllable on $J$.
\end{theorem}
\begin{proof}
   Assume that $\vartheta^\mu$ is a mild solution of \eqref{sys} in $B_r$. The stochastic Fubini theorem makes it simple to establish (at $t=T$)that
   \begin{align}
       \begin{array}{ll}\label{exp38}
           \vartheta^\mu(c) & = \tilde{\vartheta}_c -\mu(\mu I + \Delta_s^c)^{-1}[\mathbb{E}\tilde{\vartheta}_c - \Re(c)h(\vartheta^\mu)] - \displaystyle\int_0^c\mu(\mu I + \Delta_s^c)^{-1}\phi(s)\mathrm{d}\mathsf{W}(s) \\
           &\\
            &- \displaystyle\int_0^c\mu(\mu I + \Delta_0^c)^{-1}\Re(c-s)f(s,\vartheta^\mu(s))\mathrm{d}s - \int_0^c\mu(\mu I + \Delta_s^c)^{-1}\Re(c-s)f(s,\vartheta^\mu(s))\mathrm{d}\mathsf{W}(s).
       \end{array}
   \end{align}
   So, from the hypothesis \textbf{(H5)} there is a set $B_r$ such that 
   $$
   \|f(\sigma, \vartheta^\mu(\sigma))\|^2 + \|g(\sigma, \vartheta^\mu(\sigma))\|^2\leq M
   $$
   in $[0, c]\times\Omega$. Thus, there is a sub-sequence $\{f(\sigma, \vartheta^\mu(\sigma)), g(\sigma, \vartheta^\mu(\sigma))\}$, still denoted by\\ $\{f(\sigma, \vartheta^\mu(\sigma)), g(\sigma, \vartheta^\mu(\sigma))\}$, weakly convergent to, say, $\{f(\sigma), g(\sigma)\}$ in $\mathbb{H}\times L_2^0$.\\ By the notation \eqref{exp38}, we have 
  \allowdisplaybreaks
  \begin{align*}
          \mathbb{E}||\vartheta^\mu(c) - v_c||^2 &\leq 6\mathbb{E}||\mu(\mu I + \Delta_0^c)^{-1}[\mathbb{E}v_c - \Re(c)g(\vartheta^\mu)]||^2 + 6Tr(Q)\displaystyle\displaystyle\int_0^c\mathbb{E}||\mu(\mu I + \Delta_0^c)^{-1}\phi(s)||_{L_2^0}^2\mathrm{d}s \\
          &\\
           & + 6\Big(\displaystyle\int_0^c\mathbb{E}\|\mu(\mu I + \Delta_0^c)^{-1}\Re(c-s)[f(s, \vartheta^\mu(s)) - f(s)]\|\mathrm{d}s\Big)^2\\
           &\\
           & + 6\Big(\displaystyle\int_0^c\mathbb{E}\|\mu(\mu I + \Delta_0^c)^{-1}\Re(c-s)f(s)\|\mathrm{d}s\Big)^2\\
           &\\
           & + 6 Tr(\mathcal{Q})\displaystyle\int_0^c\mathbb{E}\|\mu(\mu I + \Delta_0^c)^{-1}\Re(c-s)[g(s, \vartheta^\mu(s)) - g(s)]\|^2\mathrm{d}s\\
           &\\
           & + 6 Tr(\mathcal{Q})\displaystyle\int_0^c\mathbb{E}\|\mu(\mu I + \Delta_0^c)^{-1}\Re(c-s)g(s)\|^2\mathrm{d}s.
  \end{align*}
  On the other side, according to Remark \eqref{rem41}, for all $0\leq \sigma\leq c$, we have $\mu(\mu I + \Delta_\sigma^c)^{-1}\to 0$ strongly as $\mu\to 0$, along with $||\mu(\mu I + \Delta_s^c)^{-1}||\leq 1$. Given that $\Re(\sigma), \sigma >0$ is compact and by the \textbf{Lebesgue} dominated convergence theorem, we have $\mathbb{E}||\vartheta^\mu(c) - \vartheta_c||^2\to 0$ as $\mu\to 0$. In light of this, the stochastic control system \eqref{sys} is approximately controllable on $J$.\\
   This ends the proof of this theorem.
\end{proof}
\begin{remark}
The approach employed in this article can be expanded to the investigation of the approximate controllability of a deterministic system described by,
	\begin{equation}
		\begin{cases}
			\vartheta(\sigma) = A\vartheta(\sigma) + \displaystyle\int_0^\sigma \Pi(\sigma-s)\vartheta(s)\mathrm{d}s + f_\gamma(\sigma, \vartheta(\sigma)) + \mathsf{C}u(\sigma), \sigma\in [0,c]\\
			\vartheta(0) = \displaystyle\int_0^\sigma\zeta(s, \vartheta(s))\mathrm{d}s,
		\end{cases}
	\end{equation}
	with a right choice of an abstract space $C(J, \mathbb{H})$ endowed with the norm $||\vartheta||_{C(J, \mathbb{H})}=\max_{\sigma\in J}||\vartheta(\sigma)||_\mathbb{H}$.
\end{remark}
\section{Example}\label{sec4}
The following stochastic control system is taken into consideration to highlight the major finding.
\begin{equation}\label{app}
	\begin{cases}
		\displaystyle\frac{\partial}{\partial \sigma}\omega(\sigma,x) =  \frac{\partial^2 \omega(\sigma, x)}{\partial x^2} + \displaystyle\int_0^\sigma \theta(\sigma -s)\frac{\partial^2 w(s,x)}{\partial x^2}\omega(s)\mathrm{d}s + f(\sigma, \omega(\sigma))(x) + g(\sigma, \omega(\sigma, x)\frac{\mathrm{d}\mathsf{W}(\sigma)}{\mathrm{d}\sigma}\\  \qquad \qquad \qquad +  u(x,\sigma), \sigma\in [0,c], x\in [0,1]\\
		\omega(\sigma, 0) =  
		\omega(\sigma, 1)=0,\ \sigma\in [0,c],\\
		\omega(0, x) = \displaystyle\int_0^1\Theta(s,\omega(x, s) )\mathrm{d}s,\ x\in [0, 1], 
		%\displaystyle\sum_{i=1}^n\int_0^1\Theta(x, s)\sin{\Big(\frac{v(\sigma_i, x)}{\sigma_i}\Big)}\mathrm{d}x, x\in [0,1],
	\end{cases}
\end{equation}
where $\mathsf{W}(\sigma)$ denotes the one-dimensional Brownian motion defined on the filtered probability space $(\Omega, \mathcal{F}, \{\mathcal{F}_\sigma\}_{\sigma\geq 0}, \mathbb{P})$, and $\Theta(\sigma, \omega)\in [0, c]\times L^2([0, 1])$. To write the system \eqref{app} into its abstract form like \eqref{sys}, we consider the space $\mathbb{H}=\X=\mathbb{K}= L^2([0,1])$ with the norm $\|\cdot\|$. Define the operator $A:\mathcal{D}(A)\subset \mathbb{H}\to \mathbb{H}$ by $A\omega = \omega^{\prime\prime}, \omega\in \mathcal{D}(A),$  with domain \\ 
		$$\mathcal{D}(A) =\bigg\{\omega \in\mathbb{H}\ |\ \omega, \frac{\partial \omega}{\partial x}\ \text{are absolutely continuous}, \frac{\partial^2 \omega}{\partial x^2}\in\mathbb{H},\ \omega(0)=\omega(1)=0 \bigg\}.$$
It is well known that $A$ generates a compact semigroup $T(\sigma), \sigma\geq 0$ in $\mathbb{H}$, which is supposed to be compact. This implies that the hypothesis \textbf{(R1)} is satisfied. Furthermore, $A$ has a discrete spectrum, and its eigenvalues are  $-n^2, n\in\mathbb{N}$ with
the corresponding normalized eigenvectors $e_n(x)=\sqrt{\frac{2}{\pi}}\sin(nx), 0\leq x\leq\pi, n=1,2,\cdots.$ Then
$$
A\omega = -\displaystyle\sum_{n=1}^\infty n^2\langle \omega, e_n\rangle e_n,\ \omega\in D(A), 
$$
with the associated semigroup defined by 
$$
    T(\sigma)\omega = \sum_{n=1}^{\infty}e^{-n^2\sigma}\left\langle \omega, e_n\right\rangle e_n,\quad \omega\in X.
$$
Let us define the operator $\Pi(\sigma): \mathcal{D}(A)\subset \mathbb{H}\to \mathbb{H}$ by 
$$\Pi(\sigma)\omega = \theta(\sigma)A\omega,\ \text{for}\ \sigma\geq 0\ \text{and}\ \omega\in \mathcal{D}(A).$$
Also, consider the infinite-dimensional space $\mathbb{K}$ by
$$
\mathbb{K}=\{u|u = \sum_{n=2}^\infty u_n w_n,\} \text{with}\ \sum_{n=2}^\infty u_n^2 <\infty, \quad \text{for all}\ w\in\mathbb{H}.
$$
In the space $\mathbb{K}$, the norm is defined by $\|u\|_\mathbb{K}^2=\sum_n^\infty u_n^2$. Now, the  continuous linear operator $\mathsf{C}:\mathbb{K}\to \mathbb{H}$ is defined by $\mathsf{C}u = 2u_2w_1 + \sum_{n=2}^\infty u_n w_n$.\\
\noindent Moreover, for any $\sigma\in J$, let 
\begin{align*}
     &\vartheta(\sigma)(x) = \omega(\sigma, x),\ \mathsf{C}u(\sigma)(x)=u(\sigma, x), \quad \text{for}\ \sigma\in [0, c]\ \text{and}\ x\in [0, 1]\\
     &f(\sigma, \vartheta(\sigma))(x)=\frac{\sigma \vartheta(x,\sigma)}{2(1+\vartheta^2(\sigma, x))},\quad g(\sigma, \vartheta(\sigma))(x) = \frac{1}{1 + e^\sigma}\frac{\vartheta(\sigma, x)}{1+\vartheta^2(\sigma, x)}\ \text{ and }\\
     &\zeta(\sigma,\vartheta(\sigma))(x) = \Theta(\sigma, \vartheta)=2\sigma^2\cos{(\frac{\vartheta(z, \sigma)}{\sigma})}.
     %h(v)&=\sum_{i=1}^n\Xi_h v(\sigma_i)\ \text{where}\ \Xi_h (v)(z) = \int_0^1\Theta(x, z)\sin(\frac{v(z, \sigma)}{\sigma})\mathrm{d}z\quad \text{for}\ x\in [0, 1].
\end{align*}
So, the system \textcolor{blue}{\eqref{app}} can be rewritten into the abstract form as the initial system \textcolor{blue}{\eqref{sys}}. Consequently, if $\theta:\mathbb{R}_+\to\mathbb{R}_+^*$ is $C^1$ function such that  $\theta$ and $\theta^\prime$ are bounded and uniformly continuous, we conclude that \textbf{(R2)} is satisfied.\\
Furthermore, 
\begin{align*}
     ||f(\sigma, \vartheta(\sigma))||^2 &\le \int_0^1\|\frac{\sigma \vartheta(x, \sigma)}{2(1 + \vartheta^2(x, \sigma))}\|^2\mathrm{d}x \\
     &\le \frac{\sigma^2}{4}\int_0^1\bigg\Vert \vartheta(x, \sigma)\bigg\Vert^2\mathrm{d}x\\
      &\le \frac{\sigma^2}{4}\|^2\vartheta(\sigma)\|^2.
\end{align*}
So, for all $\sigma\in [0, c], \vartheta\in\mathbb{H}$, $$
\mathbb{E}||f(\sigma,\vartheta)||^2\leq \frac{\sigma^2}{4}\mathbb{E}||\vartheta(\sigma)||^2.
$$ 
Further, 
\begin{align*}
    \mathbb{E}||g(\sigma,\vartheta)||^2&\leq \mathbb{E}\int_0^1\bigg\Vert\frac{1}{(1 + e^\sigma)}\frac{\vartheta(x,\sigma)}{(1 + \vartheta^2(x, \sigma))}\bigg\Vert^2\mathrm{d}x\\
    &\leq \bigg\Vert\frac{1}{(1 + e^\sigma)}\bigg\Vert^2 \mathbb{E}\int_0^1\bigg\Vert\frac{\vartheta(x,\sigma)}{(1 + \vartheta^2(x, \sigma))}\bigg\Vert^2\mathrm{d}x\\
    &\leq\frac{1}{4}\int_0^1\mathbb{E}\|\vartheta(x, \sigma)\|^2\mathrm{d}x\\
    &\leq \frac{1}{4}\mathbb{E}||\vartheta(\sigma)||^2,
\end{align*}
and
\begin{align*}
    \mathbb{E}||\zeta(\sigma,\vartheta)||^2 &\leq\int_0^1\mathbb{E}\|\sigma^2\sin\bigg(\frac{\vartheta(x, \sigma)}{\sigma}\bigg)\|^2\mathrm{d}x\\
    &\le \sigma^2\int_0^1\mathbb{E}\|\vartheta(x, \sigma)\|^2\mathrm{d}x\\
    & \leq \sigma^2\mathbb{E}||\vartheta(\sigma)||^2,
\end{align*}
which implies that, the \textbf{(H1)-(H3)} hold good with $\tau_f(\sigma)=\frac{\sigma^2}{4},\, \tau_g(\sigma)=\frac{1}{4},\, \tau_\zeta(\sigma)=\sigma^2$ and $\phi_f(\sigma)=\phi_g(\sigma)=\phi_\zeta(\sigma)=\sigma$. Hence, by Theorem \ref{theo31}, system \eqref{app} has a mild solution provided that \eqref{equa28} holds.\\
We use the following lemma to establish the approximate controllability of \eqref{app}. 
\begin{lemma}\cite{mokkedem2014approximate}\label{approxth}
    Consider $\theta(\sigma)\in L^1(\mathbb{R}^+)\cap C^1(\mathbb{R}^+)$ with primitive $\Pi(\sigma)\in L_{loc}^1(\mathbb{R}^+)$ such that $\Gamma(\sigma)$ is non-positive, non-decreasing and $\Pi(0)=-1$. If the operator $A$ is self-adjoint and positive semi-definite, the resolvent operator $\hat{\Re}(\sigma)$ associated with the system (\ref{cauchy}) is self-adjoint as well.
\end{lemma}
From Lemma \ref{approxth} above, the resolvent operator $\hat{\Re}(\sigma)$ of (\ref{app}) if self-adjoint. So it follows, $$\mathsf{C}^*\hat{\Re}^*(\sigma)\vartheta=\hat{\Re}(\sigma)\vartheta,\ \text{for any}\ \vartheta\in \varSigma.
$$
Let $\mathsf{C}^*\Re^*(\sigma)\vartheta = 0,$ for all $\sigma\in [0, c]$, then $\hat{\Re}(\sigma)\vartheta=0, \forall \sigma\in [0, c]$. Since $\hat{\Re}(0)=I$ for $l=0$, we get $\vartheta=0$. So from \cite{curtain2012introduction}(Theorem 4.1.7), it follows that the linear control system corresponding to (\ref{app}) is approximately controllable on $J$, which means that the Lemma \ref{lemma33} is satisfied. Therefore, by Theorems \ref{theo31}  and \ref{theo42}, the integrodifferential equation (\ref{app}) is approximately controllable on $J$.
%It can be easily seen that the deterministic linear integrodifferential corresponding to \eqref{application} is approximately controllable on $[0,1]$. 
%Additionally, if all of Theorem \ref{theo42}'s prerequisites are met, we can use this Theorem to determine that the stochastic control system \eqref{application} is approximately controllable.
\section{Conclusion}
In this article, we discuss the approximate controllability of a class of nonlinear stochastic integrodifferential equations with nonlocal initial conditions in a Hilbert space. By dropping the compactness conditions and the Lipschitz condition on the nonlocal term, we use a weaker growth condition on the nonlocal term to establish our result. Furthermore, the study of approximate controllability instead of the exact one is due to the fact that the resolvent operator related to the linear part of the main system is compact.  \\
This work can further be extended to a second order (fractional) system to study the strongest notion of controllability, called trajectory controllability, by relaxing the compactness assumption on the resolvent operators/semigroups in the form of sine and cosine operators.

{\bf Declarations:}\\
{\bf Ethical Approval:} No human and/or animal studies are involved in this research.\\
{\bf Funding:} Authors do not have any financial assistance or funding to support this research.
%\bibliographystyle{plain}
%\bibliography{biblio}

\end{document}